\documentclass{daj}

\dajAUTHORdetails{%
  title = {Bohr Sets in Sumsets I: Compact Abelian Groups}, 
  author = {Anh N. Le and Th\'ai Ho\`ang L\^e},
  plaintextauthor = {Anh N. Le and Thai Hoang Le},
  plaintexttitle = {Bohr Sets in Sumsets I: Compact Abelian Groups}, 
  runningtitle = {Bohr Sets in Sumsets I: Compact Abelian Groups}, 
  runningauthor = {Anh N. Le and Th\'ai Ho\`ang L\^e},
  copyrightauthor = {Anh N. Le and Th\'ai Ho\`ang L\^e},
  keywords = {Bohr sets, sumsets, compact groups, homomorphisms},
}   

\dajEDITORdetails{%
   year={2025},
   volume={XX},
   number={11},
   received={30 December 2021},   
   revised={19 May 2024},    
   published={4 September 2025},  
   doi={10.19086/da.143451},       
}   

\usepackage[english]{babel}
\usepackage[latin1]{inputenc}
\usepackage[normalem]{ulem}
\usepackage{amssymb}
\usepackage{amsmath}
\usepackage{amsfonts}
\usepackage{amsthm}
\usepackage{enumitem}
\usepackage{mathrsfs}
\usepackage{t1enc}
\usepackage{bbm}
\usepackage[capitalize]{cleveref}

\newtheorem{theorem}{Theorem}[section]
\newtheorem*{theorem*}{Theorem}
\newtheorem{proposition}[theorem]{Proposition}
\newtheorem{lemma}[theorem]{Lemma}

\newtheorem{question}[theorem]{Question}

\theoremstyle{definition}

\newtheorem{example}{Example}
\newtheorem*{remark*}{Remark}
\newtheorem{remark}{Remark}

\newcommand{\vertiii}[1]{{\left\vert\kern-0.25ex\left\vert\kern-0.25ex\left\vert #1 
    \right\vert\kern-0.25ex\right\vert\kern-0.25ex\right\vert}}

\newcommand{\R}{\mathbb{R}}

\newcommand{\C}{\mathbb{C}}
\newcommand{\T}{\mathbb{T}}
\newcommand{\Q}{\mathbb{Q}}
\newcommand{\Z}{\mathbb{Z}}
\newcommand{\N}{\mathbb{N}}
\newcommand{\calN}{\mathcal{N}}

\newcommand{\F}{\mathbb{F}}

\newcommand{\AN}{A^{(N)}}

\newcommand\K{{\mathbb{K}}}
\newcommand\Kinf{{\mathbb{K}_\infty}}
\newcommand\Fqt{\mathbb{F}_q((\frac{1}{t}))}

\newcommand\Fq{{\mathbb{F}_q}}
\newcommand\Fp{{\mathbb{F}_p}}

\begin{document}

\begin{frontmatter}[classification=text]

\title{Bohr Sets in Sumsets I: Compact Abelian Groups} 

\author[anhle]{Anh N. Le}

\author[thoangle]{Th\'ai Ho\`ang L\^e \thanks{Supported by National Science Foundation Grants DMS-1702296, DMS-2246921 and a travel gift from the Simons Foundation.}}

\begin{abstract}
Let $G$ be a compact abelian group and $\phi_1, \phi_2, \phi_3$ be continuous endomorphisms on $G$. Under certain natural assumptions on the $\phi_i$'s, we prove the existence of Bohr sets in the sumset $\phi_1(A) + \phi_2(A) + \phi_3(A)$, where $A$ is either a set of positive Haar measure, or comes from a finite partition of $G$. The first result generalizes theorems of Bogolyubov and Bergelson-Ruzsa. As a variant of the second result, we show that for any partition $\Z = \bigcup_{i=1}^r A_i$, there exists an $i$ such that $A_i - A_i + sA_i$ contains a Bohr set for any $s \in \Z \setminus \{ 0 \}$. The latter is a step toward an open question of Katznelson and Ruzsa.
\end{abstract}
\end{frontmatter}


\section{Introduction and statements of results}




This paper is the first one in a series in which we study Bohr sets in sumsets. The second paper in the series is titled ``Bohr sets in sumsets II: countable abelian groups'' \cite{Griesmer-Le-Le-countable}.
Let $G$ be an abelian topological group. For a finite set $\Lambda$ of characters (i.e. continuous homomorphisms from $G$ to $S^1 := \{ z \in \C: |z|=1\}$) and $\eta > 0$, the set 
\begin{equation*} \label{eq:bohr1}
B(\Lambda; \eta) := \{ x \in G : | \gamma(x)-1 | < \eta \textup{ for all } \gamma \in \Lambda\}
\end{equation*} 
is called a \textit{Bohr set} or a \emph{Bohr neighborhood of $0$}. We refer to $\eta$ as the \textit{radius} and $|\Lambda|$ as the \textit{rank} (or \textit{dimension}) of the Bohr set. The set $B(\Lambda; \eta)$ is also called a Bohr-$(|\Lambda|, \eta)$ set. 

If $A, B \subset G$, the sumset and difference set of $A$ and $B$ are $A\pm B: = \{a\pm b: a \in A, b \in B \}$. If $c \in \Z$, we define $cA:=\{ ca: a \in A\}$.
The study of Bohr sets in sumsets started with the following important theorem of Bogolyubov \cite{bogo}\footnote{This is reminiscent of Steinhaus' theorem, which says that if $A \subset \R$ has positive Lebesgue measure, then $A-A$ contains an open interval around 0.}.  

\begin{theorem}[Bogolyubov \cite{bogo}]
If $A \subset \Z$ has positive upper Banach density, i.e. 
\[
d^*(A) : = \lim_{N \rightarrow \infty} \sup_{M \in \Z} \frac{|A \cap [M+1,M+N]|}{N} >0,
\] 
then $A + A - A - A$ contains a Bohr set whose rank and radius depend only on $d^*(A)$.
\end{theorem}

While it originated from the study of almost periodic functions, Bogolyubov's theorem is now a standard tool in additive combinatorics. It was used in Ruzsa's proof of Freiman's theorem \cite{ruzsa6} and in Gowers' proof of Szemer\'edi's theorem \cite{gowers}. See \cite{bienvenu-le1, gm1} for a recent variant of Bogolyubov's theorem and its applications. 

The more copies of $A$ that are involved, the more structured the sumset is. This reflects the fact that more convolutions result in smoother functions. Thus, a natural question is: What is the smallest number of copies of $A$ that will guarantee the existence of a Bohr set? In $\Z$, it is known that $A-A$ does not necessarily contain a Bohr set when $A$ is a set of positive upper Banach density, which is a result of Kriz \cite{kriz}. On the other hand, F\o lner \cite{folner1} proved that there is a Bohr set $B$ such that $(A-A) \setminus B$ has density 0. 

Regarding three copies of $A$, Bergelson and Ruzsa \cite{br} proved the following:

\begin{theorem}[Bergelson-Ruzsa \cite{br}] \label{th:br}
Let $r, s, t$ be non-zero integers satisfying $r +s+t = 0$. If $A \subset \Z$ has positive upper Banach density, then $r A+s A+t A$ contains a Bohr set whose rank and radius depend only on $r, s, t$ and $d^*(A)$. 
\end{theorem} 
The condition $r + s + t = 0$ is easily seen to be necessary, by taking $A = M \Z + 1$ for some $M > |r| + |s| + |t|$, since any Bohr set must necessarily contain 0. In particular, one cannot expect $A + A - A$ to contain a Bohr set. When $(r, s, t) = (1, 1, -2)$, Bergelson-Ruzsa's theorem generalizes Bogolyubov's, since $A+A-2 A \subset A+A-A-A$.

\subsection{Partition results in \texorpdfstring{$\Z$}{Z}}

While the problem of finding Bohr sets in sumsets of sets having positive density has attracted much attention, the analogous question concerning partitions of $\Z$ was little studied until recently. Regarding the latter, there is a well-known problem in additive combinatorics and dynamical systems, which was popularized by Ruzsa \cite[Chapter 5]{ruzsabook} and Katznelson \cite{katznelson}. 

\begin{question} \label{q:kr}
If $\Z = \bigcup_{i=1}^r A_i$, must there exist $i \in \{1, 2, \ldots, r\}$ such that $A_i - A_i$ contains a Bohr set?
\end{question}

In terms of dynamical systems, Question \ref{q:kr} asks if any set of recurrence for minimal \textit{isometries} (also known as a set of \textit{Bohr recurrence}) is also a set of recurrence for minimal \textit{topological systems}. See \cite{gkr} for a detailed account of the history of \cref{q:kr}, as well as its many equivalent formulations.\footnote{In \cite{gkr}, what we call ``Bohr set'' is referred to as ``Bohr neighborhood of $0$.'' Furthermore, a Bohr set in their definition is an arbitrary translate of the Bohr sets defined in this paper.}
While Question \ref{q:kr} remains open at the moment and only some partial results were obtained \cite{gkr, griesmer-k}, we do have a positive answer when three copies of $A_i$ are involved. 


\begin{theorem} \label{th:main-kr}
Let $\Z = \bigcup_{i=1}^r A_i$ be a partition.
\begin{enumerate}[label=(\alph*), leftmargin=*]
    \item For any $s_1, s_2 \in \Z \setminus \{0\}$, there exists $i \in \{1, 2, \ldots, r\}$ such that
    the set $s_1 A_i -s_1 A_i + s_2A_i$ contains a Bohr set whose rank and radius depend only on $r$ and $s_1, s_2$. 
\item There exists $i \in \{1, 2, \ldots, r\}$ such that for any $s \in \Z \setminus \{0\}$, the set $A_i - A_i + s A_i$ contains a Bohr set.
\end{enumerate}
\end{theorem} 

  \cref{th:main-kr} highlights the difference between partition and density since, as we mentioned earlier, there is a set $A \subseteq \Z$ of positive density such that $A - A + A$ does not contain a Bohr set. 
  
  The expression $s_1 A_i - s_1 A_i + s_2 A_i$ is related to Rado's condition on partition regularity \cite{rado}. Recall that an equation $s_1x_1 + s_2x_2 + \cdots + s_{\ell} x_{\ell} =0$ with coefficients in $\Z \setminus \{0\}$ is \textit{partition regular} if under any finite partition (or coloring) of $\Z \setminus \{0\}$, there exists a monochromatic solution $(x_1,x_2,\ldots, x_{\ell})$. Rado's theorem says that the equation $s_1 x_1 + s_2 x_2 + \cdots + s_{\ell} x_{\ell} =0$ is partition regular if and only if $\{s_1, \ldots, s_{\ell}\}$ satisfies the following condition:
  There exists a nonempty set $J \subset \{1, \ldots, \ell\}$ such that $\sum_{i \in J} s_i =0$. 
Using the facts that $(s_1 + \cdots + s_{\ell}) A \subseteq s_1 A + \cdots + s_{\ell} A$, and a Bohr set must contain $0$, \cref{th:main-kr}(a) implies that
for $\ell \geq 3$ and $s_1, \ldots, s_{\ell} \in \Z \setminus \{0\}$, the following are equivalent:
\begin{enumerate}
 \item For any partition $\Z = \bigcup_{i=1}^r A_i$, there exists $i \in \{1, \ldots, r\}$ such that $s_1 A_i + \cdots + s_{\ell} A_i$ contains a Bohr set.
\item $\{s_1, \ldots, s_{\ell}\}$ satisfies  Rado's condition above.
\end{enumerate}

A novelty of \cref{th:main-kr}(b) is that it guarantees a single set $A_i$ that works for every coefficient $s$ (on the other hand, we do lose control on the rank and radius of the Bohr set). When $s$ is very large, the set $sA_i$ is small and so its contribution to the sum diminishes. While there is no consensus on what the answer to \cref{q:kr} should be, \cref{th:main-kr}(b) provides evidence that the answer to \cref{q:kr} is either positive or very delicate.

In \cite[Table 1, p. 8]{gkr}, Glasscock-Koutsogiannis-Richter summarized results on Bohr sets in sumsets, pertaining to both density and partition. Our \cref{th:main-kr} fills in the blank on the Syndeticity\footnote{A subset $A$ of a group $G$ is called \textit{syndetic} if $G$ can be covered by finitely many translates of $A$.} column and $rA + sA+ tA$ row of their table.

In forthcoming work \cite{lw}, Wathodkar and the second author extend part (b) of Theorem \ref{th:main-kr} to accomodate more general sumsets. Their result says that for any partition $\Z = \bigcup_{i=1}^r A_i$, there exists $i \in \{1, 2, \ldots, r\}$ such that for any $s_1, s_2 \in \Z \setminus \{0\}$, the set $s_1 A_i - s_1 A_i + s_2 A_i$ contains a Bohr set.


\subsection{Results in compact groups}

Bogolyubov's theorem has been generalized to other groups as well (in more general groups, the upper Banach density $d^*$ can be defined in terms of  \textit{F\o lner sequences} or \textit{invariant means}). F\o lner \cite{folner1, folner2} extended Bogolyubov's theorem to all abelian groups. Answering a question of Hegyv\'ari-Ruzsa \cite{hr}, Bj\"orlund-Griesmer \cite{bg} proved that in any countable discrete abelian group $G$, for any $A \subset G$ with $d^*(A) >0$, for ``many'' $a \in A$, the set $A+A-A-a$ contains a Bohr set whose rank and radius depend only on $d^*(A)$. Very recently, Griesmer \cite{griesmer-br} generalized Theorem \ref{th:br} to all countable discrete abelian groups, though his proof does not give effective bounds for the rank and radius of the Bohr set in question.


Bergelson-Ruzsa and Bogolyubov first proved their theorems in the cyclic group $\Z_N$, and the statements in $\Z$ follow from a compactness argument. Likewise, in Bj\"orlund-Griesmer \cite{bg} and Griesmer \cite{griesmer-br}, certain compact groups (namely \textit{Bohr compactifications} and \textit{Kronecker factors}) play a prominent role. In view of this ``compact first'' strategy, the main goal of this paper is, in fact, to study the existence of Bohr sets in sumsets of compact groups. Under this investigation, \cref{th:main-kr} arises as an application of our general method. 

Another feature of our work is the consideration of continuous homomorphisms $\phi:G \rightarrow G$ and the image $\phi(A)$ rather than just dilations $c A$. This point of view leads to a wider range of applications, for example, linear maps on vector spaces and multiplication by an element in a ring (see Theorems \ref{th:nf} and \ref{th:ff} below). This new perspective was also adopted in recent work of Ackelsberg-Bergelson-Best \cite{abb} on Khintchine-type recurrence for actions of an abelian group (Theorem \ref{th:abb} below). 

Our main result on Bohr sets in sumsets arising from partitions is as follows. 


\begin{theorem} \label{th:main-partition}
Let $G$ be a compact abelian group with normalized Haar measure $\mu$ and let $\phi_1, \phi_2: G \rightarrow G$ be continuous homomorphisms satisfying
\begin{enumerate}[label=(\alph*), leftmargin=*]
\item $\phi_1, \phi_2$ are commuting, and
\item $\phi_1(G), \phi_2(G)$ have finite index in $G$.
\end{enumerate} 
Let $G = \bigcup_{i=1}^r A_i$ be a partition of $G$ into measurable sets. Then for some $1 \leq i \leq r$, 
\[
\phi_1(A_i) - \phi_1(A_i) + \phi_2(A_i)
\] 
contains a Bohr-$(k, \eta)$ set, where $k$ and $\eta$ depend only on $r$, $[G:\phi_1(G)]$ and $[G:\phi_2(G)]$.
\end{theorem}
\begin{remark}\label{remark:first_one} \
\begin{itemize}[leftmargin=*]
\item If $\mu(A) >0$, then $A + A - A$ is not guaranteed to contain a Bohr set. For a counterexample, take $G = \Z_{2N}$ for some large $N$ and $A=\{a \in \Z_{2N}: a \textup{ is odd} \}$. In particular, the analogous version of \cref{th:main-partition} for sets of positive measure fails. 

\item We do not know if the commuting condition can be removed entirely, though it can be slightly relaxed (see \cref{th:main-partition-true}). For example, commutativity is not required when $\phi_1$ or $\phi_2$ is an automorphism (see Remark \ref{rk:auto2}).

\item The finite index condition on $\phi_1(G)$ cannot be removed, by taking for example $\phi_1 = 0$ and $\phi_2(x) = x$. On the other hand, we do not know whether the finite index condition on $\phi_2(G)$ can be removed. If we let $\phi_2 = 0$, then the situation amounts to \cref{q:kr} itself. 

\end{itemize}
\end{remark}

We now turn our attention to density results. In compact abelian groups, the Haar measure plays the role of the upper Banach density.  

\begin{theorem} \label{th:main-density}
Let $G$ be a compact abelian group with normalized Haar measure $\mu$ and $\phi_1, \phi_2, \phi_3: G \rightarrow G$ be continuous homomorphisms satisfying
\begin{enumerate}[label=(\alph*), leftmargin=*]
\item $\phi_1 + \phi_2 + \phi_3 = 0$,
\item $\phi_1, \phi_2, \phi_3$ are commuting,
\item $\phi_1(G), \phi_2(G), \phi_3(G)$ have finite index in $G$.
\end{enumerate} 
Let $A \subseteq G$ be a measurable subset with $\mu(A) = \delta >0$. Then 
\[
\phi_1(A) + \phi_2(A) + \phi_3(A)
\] contains a Bohr-$(k, \eta)$ set, where $k$ and $\eta$ depend only on $\delta$ and the indexes $[G:\phi_i(G)]$ $(1 \leq i \leq 3)$.
\end{theorem}

\begin{remark}\label{remark:density_commutative}\
\begin{itemize}[leftmargin=*]
\item The condition $\phi_1 + \phi_2 + \phi_3 = 0$ cannot be removed. For a counterexample, take $G=\Z_N$ for some large $N$ and $A = \{1, \cdots, \lfloor N/10 \rfloor\}$. Then $A+A+A$ does not contain 0, and hence is not a Bohr set.
\item We do not know if the condition on commutativity can be removed entirely, though it can be weakened (see Theorem \ref{th:main-density-true}). For example, commutativity is not required when one of $\phi_1, \phi_2$ and $\phi_3$ is an automorphism (see Remark \ref{rk:auto}).
\item The finite index condition cannot be removed. Indeed, we can take $G = \F_2^n$ for some large $n$, $\phi_1(x)= x$, $\phi_2(x)=-x, \phi_3(x)=0$. In this setting, Bohr sets are simply vector subspaces. A construction of Green \cite[Theorem 9.4]{green-ff} gives a set $A$ of size $\geq |G|/4$ such that any subspace contained in $A-A$ must have codimension $\geq \sqrt{n}$.  
\end{itemize}
\end{remark}

In part II of this series \cite{Griesmer-Le-Le-countable} with Griesmer, by means of transference principles, we prove analogues of Theorems \ref{th:main-density} and \ref{th:main-partition} for countable discrete abelian groups. 
In particular, we obtain an effective version of the aforementioned result of Griesmer \cite{griesmer-br}. 

\subsection{Number-theoretic consequences} 
As mentioned earlier, the fact that we accommodate homomorphisms in \cref{th:main-partition} and \cref{th:main-density} enables us to generalize \cref{th:br} and \cref{th:main-kr} to number fields and function fields. 

In the following, for a subset $A$ of a ring $R$ and $c \in R$, we write 
\begin{equation} \label{eq:ring1}
cA = \{ca: a \in A\}
\end{equation}
and
\begin{equation} \label{eq:ring2}
    A/c = \{ b\in R: bc \in A\}.
\end{equation}
The next theorem is true for any number field, but we only state for $\Z[i]$ for simplicity.

\begin{theorem} \label{th:nf}
    Let $s_1, s_2, s_3 \in \Z[i] \setminus \{0\}$ such that $s_1 + s_2 + s_3 = 0$. 
		\begin{enumerate}[label=(\alph*), leftmargin=*]
\item		If a set $A \subseteq \Z[i]$ has positive upper density, i.e.
    \[
        \overline{d}(A) := \limsup_{N \to \infty} \frac{|A \cap [-N, N]^d|}{(2N + 1)^d} = \delta > 0,
    \]
    then $s_1 A + s_2 A + s_3 A$ contains a $(k, \eta)$-Bohr set in $\Z[i]$, where $k$ and $\eta$ depend only on $s_1, s_2, s_3$ and $\delta$. 
		
\item If $\Z[i] = \bigcup_{j=1}^r A_j$, then for some $j \in \{1, 2, \ldots, r\}$, $s_1 A_j - s_1 A_j + s_2 A_j$ contains a $(k, \eta)$-Bohr set in $\Z[i]$, where $k$ and $\eta$ depend only on $s_1, s_2$ and $r$.

\item  If $\Z[i] = \bigcup_{j=1}^r A_j$, then there exists $j \in \{1, 2, \ldots, r\}$ such that $A_j - A_j + s A_j$ contains a Bohr set for any $s \in \Z[i] \setminus \{0\}$.

\end{enumerate}	
Here, as a group, we identify $\Z[i]$ with $\Z^2$.
\end{theorem}

Our next result deals with 
the ring $\F_q[t]$ of polynomials over a finite field $\Fq$.

\begin{theorem} \label{th:ff}
    Let $s_1, s_2, s_3 \in \Fq[t] \setminus \{0\}$ such that $s_1 + s_2 + s_3 = 0$. 
		
	\begin{enumerate}[label=(\alph*), leftmargin=*]
		\item If a set $A \subseteq \Fq[t]$ has positive upper density, i.e. 
    \[
        \overline{d}(A) := \limsup_{N \to \infty} \frac{|\{x \in \Fq[t]: \deg x < N \}|}{q^N} = \delta > 0,
    \]
    then $s_1 A + s_2 A + s_3 A$ contains a $\Fq$-vector subspace of finite codimension of $\F_q[t]$, where the codimension depends only on $s_1, s_2, s_3$ and $\delta$. 
    
\item If $\Fq[t] = \bigcup_{i=1}^r A_i$, then for some $i \in \{1, \ldots, r\}$, $s_1 A_i - s_1 A_i + s_2 A_i$ contains a $\Fq$-vector subspace of finite codimension of $\F_q[t]$, where the codimension depends only on $s_1, s_2$ and $\delta$. 

\item If $\F_q[t] = \bigcup_{i=1}^r A_i$, then there exists $i \in \{1, 2, \ldots, r\}$ such that $A_i - A_i + s A_i$ contains an $\Fq$-vector subspace of finite codimension of $\F_q[t]$ for any $s \in \Fq[t] \setminus \{0\}$.
\end{enumerate}
\end{theorem}

We remark that the special case $s_1, s_2, s_3 \in \Fq \setminus\{0\}$ of  \cref{th:ff}(a) is essentially Corollary 1.4 in \cite{griesmer-br}.

\subsection{Counting linear patterns}
Similarly to the proofs of Bogolyubov \cite{bogo} and Bergelson-Ruzsa \cite{br}'s theorems, we deduce Theorem \ref{th:main-density} from a lower bound (of correct order of magnitude) for the number of certain linear patterns in $G$. This is straightforward in Bogolyubov's case, but less so in Bergelson and Ruzsa's. Bergelson and Ruzsa had to count the number of generalized Roth patterns $\{ x, x+ry, x+sy \}$ (where $r, s \in \Z$) and they deduced this from Szemer\'edi's theorem \cite{szemeredi} and Varnavides' argument \cite{var}. For us, we need to count the number of patterns $\{ x, x+ \phi(y), x+ \psi(y) \}$ (where $\phi$ and $\psi$ are homomorphisms). This is accomplished by generalizing a Fourier-analytic argument of Bourgain \cite{bourgain-roth}. Bourgain's argument, in essence an arithmetic regularity lemma, allows us to obtain the following Khintchine-type result.
 
\begin{theorem}[Khintchine-Roth theorem in compact abelian groups]
\label{th:roth-khintchine}
Let $G$ be a compact abelian group with  probability Haar measure $\mu$ and $\phi, \psi: G \to G$ be continuous homomorphisms such that $[G: \phi(G)], [G: \psi(G)]$ and $[G: (\phi-\psi)(G)]$ are finite. 
Let $f: G \to [0, 1]$ be a measurable function with $\int_G f \, d \mu = \delta >0$.
 
Then for any $\epsilon > 0$, there exists a constant $c_1 >0$ that depends only on $\delta, \epsilon$ and the indexes above such that the set
\[
    B = \left\{y \in G: \int_G f(x) f(x+ \phi(y)) f(x + \psi(y)) \, d \mu(x) > \delta^3 - \epsilon \right\}
\]
has measure at least $c_1$. Consequently,
\begin{equation} \label{eq:counting-roth}
\iint_{G^2} f(x) f(x+ \phi(y)) f(x + \psi(y)) \, d \mu(x) d\mu(y) \geq c_2 
\end{equation}
for some positive constant $c_2$ depending only on $\delta$ and the indexes above.
\end{theorem}

Theorem \ref{th:roth-khintchine} was proved independently by Berger-Sah-Sawhney-Tidor \cite{bsst}, under the hypothesis that $\phi, \psi$  and $\phi-\psi$ are automorphisms, using a very similar argument. 
Our execution is slightly different from theirs, in that we follow Bergelson-Host-McCutcheon-Parreau \cite{bhmp}'s elaboration of Bourgain's argument, while they follow Tao \cite{tao}'s.

Theorem \ref{th:roth-khintchine} is markedly similar to the following result of Ackelsberg, Bergelson and Best: 

\begin{theorem}[{\cite[Theorem 1.10]{abb}}] \label{th:abb} Let $G$ be a countable discrete abelian group, and $\phi, \psi: G \to G$ be homomorphisms such that $[G: \phi(G)], [G: \psi(G)]$ and $[G: (\phi-\psi)(G)]$ are finite. For any ergodic system $(X,\mathcal{B}, \mu, (T_g)_{g \in G})$, any $\epsilon > 0$, and any $A \in \mathcal{B}$, the set
\[
B = \left\{ g \in G : \mu ( A \cap T^{-1}_{\phi(g)} A \cap T^{-1}_{\psi(g)} A ) > \mu(A)^3 - \epsilon \right\}
\] is syndetic in $G$. 
\end{theorem}

As discussed in \cite[Section 10]{abb}, the finite index condition in Theorem \ref{th:abb} is necessary. The following result of Fox-Sah-Sawhney-Stoner-Zhao \cite{fsssz}, improving on an earlier result of Mandache \cite{mandache}, shows that the finite index condition is also necessary in Theorem \ref{th:roth-khintchine}. 

\begin{example} \label{ex:roth-khintchine}
Let $\ell <4$ be arbitrary and $\delta >0$ be sufficiently small in terms of $l$. Let $G = \F_2^n \times \F_2^n$ where $n$ is sufficiently large, $\phi(u,v)=(u,0), \psi(u,v)=(0,u)$. Then the left hand side of \eqref{eq:counting-roth} counts the number of ``corners'' $\{ (a, b) , (a+u, b), (a, b+u) \}$ in $\F_2^n \times \F_2^n$. \cite[Corollary 1.3]{fsssz} states that there exists a set $A \subset G$ of size $\geq \delta |G|$ such that for any $u \in \F_2^n \setminus \{0\}$, we have
\[
\# \{ (a, b) \in G: (a, b) , (a+u, b), (a, b+u) \in A \} < \delta^\ell |G|.
\]
Hence, the set $B$ in Theorem \ref{th:roth-khintchine} has to be $\{ 0 \} \times \F_2^n$. But the measure of this set in $G$ goes to $0$ as $n$ goes to infinity.
\end{example}

Regarding Theorem \ref{th:main-partition}, we deduce it from the following result, which counts the number of monochromatic configurations under finite partitions of $G$.

\begin{theorem} \label{th:counting-partition}
Let $G$ be a compact abelian group with  probability Haar measure $\mu$ and let $\psi, \phi_1, \ldots, \phi_k: G \to G$ be continuous homomorphisms satisfying:
\begin{enumerate}[label=(\alph*), leftmargin=*]
\item $\psi, \phi_1, \ldots, \phi_k$ are commuting, and
\item $\psi(G), \phi_1(G), \ldots, \phi_k(G)$ have finite index in $G$.
\end{enumerate}
Suppose $G = \bigcup_{i=1}^r A_i$ is a partition of $G$ into measurable sets. Then 
\begin{equation} \label{eq:counting-brauer}
\sum_{i=1}^r \iint_{G^2} 1_{A_i}(\psi(y)) 1_{A_i}(x) 1_{A_i}(x + \phi_1(y)) \cdots 1_{A_i}(x + \phi_k(y)) \, d \mu(x) d\mu(y) \geq c_3
\end{equation} 
for some positive constant $c_3$ depending only on $r, k$ and the indexes above.
\end{theorem}

\begin{remark}\
\begin{itemize}[leftmargin=*]
\item By taking $\psi =0$, we see that the condition $[G:\psi(G)]$ is finite cannot be removed. However, we do not know whether the condition $[G: \phi_i(G)] < \infty$ is necessary or not. 
\item Our proof relies heavily on the commuting condition and we do not know if it can be removed.
\end{itemize}
\end{remark}

When $\psi$ and $\phi$ are dilations, the configuration $\{ \psi(y), x, x + \phi_1(y), \ldots, x + \phi_k(y) \}$ becomes the \textit{Brauer configuration} $\{y, x, x+y, \ldots, x + ky\}$. Results on counting such monochromatic configurations have been established by Serra-Vena \cite[Theorem 1.3]{sv} for finite abelian groups of bounded torsion. Thus, besides the fact that it allows for more general homomorphisms, Theorem \ref{th:counting-partition} has the advantage of being uniform over all groups. On the other hand, our finite index condition is certainly related, and in a sense, dual to Serra-Vena's bounded exponent condition \cite{sv}.

We remark that despite the apparent similarity between \eqref{eq:counting-roth} and \eqref{eq:counting-brauer}, their proofs are very different. The proof of \cref{th:counting-partition} is ``Fourier-free'' 
and its main ingredient is the Hales-Jewett theorem. Thus, our approach in proving this theorem is also genuinely different from Serra-Vena's, which relies on a removal lemma for groups. 

On the quantitative side, our bounds leave much to be desired. Since the proof of \cref{th:roth-khintchine} relies on the regularity lemma (\cref{prop:regularity_lemma}), in \cref{th:main-density}, the dependence of $k$ and $\eta$ on $\delta$ and $[G:\phi_i(G)]$ is of tower type. Likewise, since the proof of \cref{th:counting-partition} uses the Hales-Jewett theorem, the bounds for $k$ and $\eta$ in \cref{th:main-partition} are even worse. It is an interesting problem to obtain good bounds for Theorems \ref{th:main-density} and \ref{th:main-partition}, even in special classes of groups such as $\F_p^n$. Indeed, Sanders \cite[Theorem A.1]{sanders} obtained a near optimal bound for Bogolyubov's theorem in $\F_p^n$.

\textbf{Outline of the paper.}
In \cref{sec:prelim}, we set up notation and collect some basic facts about Bohr sets, kernels and homomorphisms in compact abelian groups. \cref{sec:partition} is devoted to proving results involving partitions, especially, Theorems \ref{th:main-partition} and \ref{th:counting-partition}. Theorems \ref{th:main-density}, \ref{th:roth-khintchine} and related density results will be proved in \cref{sec:density}. \cref{sec:z_and_field} contains proofs of results in $\Z$, number fields and function fields, i.e. Theorems \ref{th:main-kr}, \ref{th:nf} and \ref{th:ff}. Lastly, we present some related open questions in \cref{sec:open_question}.

\section{Preliminaries}
\label{sec:prelim}
In this section, we gather some background on Bohr sets, kernels and homomorphisms in compact abelian groups. Most of the results are well-known or resemble known theorems. We include proofs for the results that we cannot pinpoint precisely in the literature. 


\subsection{Notation} We write $[N]$ for the set $\{ 1, \ldots, N\}$. If $A$ and $B$ are two quantities, we write $A = O(B)$ or $A \ll B$ if there is a constant $C$ such that $|A| \leq CB$. We write $e(x)$ for $e^{2 \pi i x}$.

Throughout this paper, $G$ is a Hausdorff compact abelian group with probability Haar measure $\mu$ and  $\Gamma$ is the dual of $G$, written additively. The relevance of homomorphisms is that if $\gamma \in \Gamma$ and $\phi:G \rightarrow G$ is a continuous homomorphism, then $\gamma \circ \phi$ is also an element of $\Gamma$.

If $f: G \rightarrow \C$ is a function, for $t \in G$ we define the function $f_t(x) = f(x+t)$. 
For $f \in L^1(G)$, the Fourier transform of $f$ is the function
\[
\widehat{f} (\gamma) = \int_{G} f(x) \overline{\gamma (x)} \, d\mu(x) \qquad \textup{ for } \gamma \in \Gamma.
\]
For $f, g \in L^2(G)$, we then have Parseval's formula
\[
\int_{G} f(x) \overline{g(x)} \, d\mu(x) = \sum_{\gamma \in \Gamma} \widehat{f} (\gamma) \overline{ \widehat{g} (\gamma) }
\]
and Parseval's formula
\[
\int_{G} |f(x)|^2 \, d\mu(x) = \sum_{\gamma \in \Gamma} \left| \widehat{f} (\gamma) \right|^2.
\]

\subsection{Bohr sets} 
For $\Lambda, \Lambda_1, \Lambda_2 \subseteq \Gamma$ and $\eta_1, \eta_2 > 0$, it follows from the definition of Bohr sets that 
\[
    B(\Lambda_1; \eta_1) \cap B(\Lambda_2; \eta_2) \supset B(\Lambda_1 \cup \Lambda_2; \min(\eta_1, \eta_2))
\]
and
\[
    B(\Lambda; \eta_1) + B(\Lambda; \eta_2) \subset B(\Lambda; \eta_1 + \eta_2).
\]

\begin{lemma} \label{lem:translate}
Suppose $f_1, \ldots, f_k \in L^\infty(G)$, $\| f_i \|_{\infty} \leq 1$ for all $i=1, \ldots, k$. Let $\phi_1, \ldots, \phi_k$ be continuous homomorphisms $G \rightarrow G$. Then for any $\eta >0$, the set
\[
B = \{ t \in G: \| \widehat{f_i} - \widehat{f_{i, \phi_i(t)}} \|_{\infty} < \eta \textup{ for }i=1, \ldots, k \}
\]
contains a Bohr set $B(\Lambda; \eta)$ where $|\Lambda| \leq \frac{4k}{\eta^2}$.
\end{lemma}
\begin{proof}
Note that if $\| \widehat{f_i} - \widehat{f_{i, \phi_i(t)}} \|_\infty \geq \eta$, then for some $\gamma \in \Gamma$, 
\[
|\widehat{f_i}(\gamma) - \widehat{f_{i, \phi_i(t)}}(\gamma)| = |1 - \gamma(\phi_i(t))| |\widehat{f_i}(\gamma)| \geq \eta.
\]
This implies that $|1 - \gamma(\phi_i(t))| \geq \eta$ and $\gamma \in \Lambda_i := \{ \lambda \in \Gamma : |\widehat{f_i}(\lambda)| \geq \eta /2 \}$.

We have thus shown that 
\[
  B( \bigcup_{i=1}^k \Lambda_i \circ \phi_i; \eta) \subset B,  
\]
where $\Lambda_i \circ \phi_i: = \{ \gamma \circ \phi_i: \gamma \in \Lambda_i\} \subset \Gamma$. By  Parseval's formula, 
\[
\left( \frac{\eta}{2} \right)^2 |\Lambda_i| \leq \sum_{\lambda \in \Lambda_i} \left|\widehat{f_i}(\lambda)\right|^2 \leq 1.
\]
Therefore, $|\Lambda_i| \leq \frac{4}{\eta^2}$ and $|\bigcup_{i=1}^k \Lambda_i \circ \phi_i| \leq \frac{4k}{\eta^2}$.
\end{proof}

Next lemma is needed in \cref{sec:z_and_field}.

\begin{lemma} \label{lem:subgroup-bohr}
Let $H$ be a locally compact abelian group, $K$ be a closed subgroup of finite index $m$. Then $K$ is a Bohr-$(m, |e(1/m)-1|)$ set in $H$.
\end{lemma}
\begin{proof}
Let $\chi_1, \ldots, \chi_m$ be all characters on $H/K$. For any $x \in H/K$ and $1 \leq i \leq m$, we have $|\chi_i(x)|^m = 1$, so either $\chi_i(x) = 1$ or $|\chi_i(x) -1| \geq |e(1/m)-1|$. If $\chi_i(x) = 1$ for all $i$ then $x=0$. Hence
\[
\{ 0 \} = B(\chi_1, \ldots, \chi_m; |e(1/m)-1|).
\]
The characters $\chi_i$ lift to characters $\tilde{\chi_i}$ on $H$ by $\tilde{\chi_i} (h) = \chi_i(h+K)$. Therefore,
\[
K = B(\tilde{\chi_1}, \ldots, \tilde{\chi_m}; |e(1/m)-1|),
\]
as desired.
\end{proof}

We will also need Bogolyubov's theorem for compact abelian groups.

\begin{lemma}[Bogolyubov for compact abelian groups, see {\cite[Lemma 2.1]{ruzsa6}}]
\label{lem:bogolyubov_compact_group}
Let $G$ be a compact abelian group with Haar measure $\mu$ and let $A \subseteq G$ of positive measure. Then $A - A + A - A$ contains a Bohr-$(k, \eta)$ set where $k, \eta$ depends only on $\mu(A)$.
\end{lemma}

\subsection{Kernels}
A kernel on $G$ is a non-negative continuous function that satisfies $\int_G K \, d\mu =1$. In our case, we will utilize the kernels supported on given Bohr sets whose Fourier transforms are non-negative. For a kernel $K$, we write $\lVert \widehat{K} \rVert_1$ to denote $\sum_{\gamma \in \Gamma} | \widehat{K}(\gamma)|$.  

\begin{lemma}[cf. {\cite[Lemma 4.3]{bhmp}}]
\label{lem:kernel_a}
    Given a finite set $\Lambda \subset \Gamma$ and $\eta \in (0, 1/2]$, there exists a kernel $K$ satisfying the following:
    \begin{enumerate}
        \item $K \geq 0, \widehat{K} \geq 0$ and $\int_G K \, d \mu = \lVert K \rVert_1 = 1$,
        \item $\lVert \widehat{K} \rVert_1 = \lVert K \rVert_{\infty} \leq 1/(C_0\eta)^{|\Lambda|}$ for some absolute constant $0 < C_0 \leq 1$, and         
        \item $K$ vanishes outside the Bohr set $B( \Lambda; \eta)$.
    \end{enumerate}
Consequently, 
\begin{equation} \label{eq:mub}
\mu(B( \Lambda; \eta)) \geq (C_0\eta)^{|\Lambda|}.
\end{equation} 
\end{lemma}
We remark that the bound \eqref{eq:mub} can also be obtained from an elementary covering argument (see \cite{tao}). 
\begin{proof}
    First, for each $\lambda \in \Lambda$, there exists a kernel $K_{\lambda}: G \to [0, \infty)$ satisfying the following properties:
\begin{enumerate}
\item $\lVert K_{\lambda} \rVert_1 = 1$,
\item $\widehat{K}_{\lambda} \geq 0$,
\item $K_{\lambda}$ is supported on $B(\{\lambda\}; \eta) = \{x \in G: |\lambda(x) - 1| < \eta\}$,
\item $\|K_\lambda\|_{\infty} = K_{\lambda}(0) \leq 1/(C_0 \eta)$ \text{ for some absolute constant $0 < C_0 \leq 1$}.
\end{enumerate}		
Indeed, let $B = B(\{\lambda\}; \frac{\eta}{2})$ and let $K_{\lambda} = \frac{1_B}{\mu(B)}*\frac{1_B}{\mu(B)}$. Clearly the first and second properties are satisfied. Additionally, $K_{\lambda}$ is supported on $B(\{\lambda\}; \frac{\eta}{2}) + B(\{\lambda\}; \frac{\eta}{2}) \subset B(\{\lambda\}; \eta)$. 

Concerning the last property, we have for every $x \in G$,
\[
    K_{\lambda}(x) = \sum_{\gamma \in \Gamma} \widehat{K}_{\lambda}(\gamma) \gamma(x)
\]
and so
\[
    \left| K_{\lambda}(x) \right| \leq \sum_{\gamma \in \Gamma} \widehat{K}_{\lambda}(\gamma) = K_{\lambda}(0).
\]
Therefore, $\lVert K_{\lambda} \rVert_{\infty} = K_{\lambda}(0) = \frac{1}{\mu(B)}$.
Since $\lambda$ is continuous, its image $\lambda(G)$ is a closed subgroup of $S^1 = \{ z \in \C: |z|=1\}$, and so it is either $S^1$ or $\{ z \in \C : z^q=1\}$ for some $q \in \N$. Since $\lambda$ is a homomorphism, it is measure-preserving (see Lemma \ref{lem:pushforward} below). Hence $\mu(B)$ is equal to the normalized Haar measure of the set
\[
\left\{ z \in S^1 : |z-1| < \frac{\eta}{2} \right\}
\]
in the group $\lambda(G)$. In either case, where $\lambda(G)=S^1$ or $\{ z \in \C : |z|^q=1\}$, we find that $\mu(B) \geq C_0 \eta$ for some absolute constant $0 < C_0 \leq 1$. Therefore, $\lVert K_{\lambda} \rVert_{\infty} \leq 1/(C_0 \eta)$.  
	
We now define
    \[
        \widetilde{K} = \prod_{\lambda \in \Lambda} K_{\lambda}.
    \]
    It follows that $\widetilde{K} \geq 0$ and $\widetilde{K}$ is supported on $B(\Lambda; \eta)$. Repeatedly using the fact that $\widehat{fg}(\gamma) = \sum_{\lambda \in \Gamma} \widehat{f}(\lambda) \widehat{g}(\gamma - \lambda)$ for all $f, g \in L^{\infty}(G)$, we have $\widehat{\widetilde{K}} \geq 0$. Likewise, since $\widehat{K}_{\lambda}(0) = \lVert K_{\lambda} \rVert_1 = 1$, we have $\lVert \widetilde{K} \rVert_1 = \int_G \widetilde{K} \ d \mu =  \widehat{\widetilde{K}}(0) \geq 1$.
    
    For every $x \in G$, $\widetilde{K}(x) = \sum_{\gamma \in \Gamma} \widehat{\widetilde{K}}(\gamma) \gamma(x)$ and so 
\[
    |\widetilde{K}(x)| \leq \sum_{\gamma \in \Gamma} |\widehat{\widetilde{K}}(\gamma)| = \lVert \widehat{\widetilde{K}} \rVert_1.
\]
It follows that $\lVert \widetilde{K} \rVert_{\infty} \leq \lVert \widehat{\widetilde{K}} \rVert_1$.
Moreover,
\[
    \widetilde{K}(0) = \sum_{\gamma \in \Gamma} \widehat{\widetilde{K}}(\gamma) \gamma(0) = \sum_{\gamma \in \Gamma} \widehat{\widetilde{K}}(\gamma) = \lVert \widehat{\widetilde{K}} \rVert_1
\]
because $\widehat{\widetilde{K}}(\gamma) \geq 0$ for all $\gamma$. Thus, $\lVert \widetilde{K} \rVert_{\infty} = \lVert \widehat{\widetilde{K}} \rVert_1$. Upon defining $K = \widetilde{K}/\lVert \widetilde{K} \rVert_{1}$,
    we obtain the desired kernel.
\end{proof}

\subsection{Homomorphisms} We will often make use of the following facts about homomorphisms $G \rightarrow G$.

\begin{lemma}\label{lem:homo-solutions}
Let $\phi: G \rightarrow G$ be a continuous homomorphism such that $[G:\phi(G)] = m$ is finite. Then for any $\gamma \in \Gamma$, there are at most $m$ elements $\chi \in \Gamma$ such that $\gamma = \chi \circ \phi$.
\end{lemma}
\begin{proof}
It is easy to see that for each $\gamma \in \Gamma$, the set $S_\gamma := \{ \chi \in \Gamma : \gamma = \chi \circ \phi\}$ is either empty, or a coset of the group $S_0$. On the other hand, $S_0$ is the annihilator of the group $\phi(G)$, so by \cite[Theorem 2.1.2]{rudin}, it is isomorphic to $G/\phi(G)$, and hence has cardinality $m$.
\end{proof}

\begin{lemma} \label{lem:composition}
Let $\phi, \psi: G \rightarrow G$ be  homomorphisms such that $[G:\phi(G)] = m$ and $[G:\psi(G)] = \ell$ are finite. Then
$[G: \phi ( \psi (G))] \leq m\ell$ is finite.
\end{lemma}
\begin{proof}
We have $[G:\phi(\psi(G))] = [G:\phi(G)] [\phi(G): \phi(\psi(G))]$. It suffices to show that $[\phi(G): \phi(\psi(G))] \leq \ell$. 

Let $x_1 + \psi(G), \ldots, x_\ell + \psi(G)$ be all cosets of $\psi(G)$ in $G$. Then $\phi(x_1) + \phi(\psi(G)), \ldots, \phi(x_\ell) + \phi(\psi(G))$ are all cosets of $\phi(\psi(G))$ in $\phi(G)$ (these are not necessarily distinct, so the actual number of cosets may be less than $\ell$), proving the desired claim. 
\end{proof}

\begin{lemma}
\label{lem:pushforward}
    Let $G, H$ be compact abelian groups and $\mu, \nu$ be the normalized Haar measures of $G$ and $H$, respectively. Suppose $\phi: G \to H$ is a continuous surjective homomorphism. Then $\phi_* \mu = \nu$ (i.e. $\nu(B) = \mu( \phi^{-1}(B))$ for any Borel set $B \subset H$).
\end{lemma}
\begin{proof}
    Let $\nu_0 = \phi_* \mu$. By the uniqueness of the normalized Haar measure, it suffices to show that $\nu_0$ is a translation-invariant probability measure on $H$. First, $\nu_0$ is a probability measure because $\nu_0(H) = \mu(\phi^{-1}(H)) = \mu(G) = 1$. Now let $B \subset H$ be a Borel set and $h_0 \in H$ be arbitrary. Since $\phi$ is surjective, there exists $g_0 \in G$ such that $\phi(g_0) = h_0$. For any $g \in \phi^{-1}(B + h_0)$, we have
    \[
        \phi(g - g_0) = \phi(g) - \phi(g_0) \in B + h_0 - h_0 = B.
    \]
    Therefore, $\phi^{-1}(B + h_0) \subseteq \phi^{-1}(B) + g_0.$ On the other hand, 
    \[
        \phi(\phi^{-1}(B) + g_0) \subseteq B + h_0
    \]
    and so $\phi^{-1}(B + h_0) = \phi^{-1}(B) + g_0$. Since $\mu$ is translation-invariant on $G$, it follows that
    \[
        \nu_0(B + h_0) = \mu(\phi^{-1}(B + h_0)) = \mu(\phi^{-1}(B) + g_0) = \mu(\phi^{-1}(B)) = \nu_0(B).
    \]
    Thus $\nu_0$ is translation-invariant on $H$ and so $\nu_0 = \nu$. 
\end{proof}

\begin{lemma} \label{lem:finite-index}
Let $\phi:G \rightarrow G$ be a continuous homomorphism such that $[G:\phi(G)]=m$ is finite. Then for any measurable set $A \subset G$, we have 
\begin{equation}
\mu( A ) \leq m \mu( \phi(A) )
\end{equation}
and
\begin{equation}
\mu( \phi^{-1}(A) ) \leq m \mu( A ).
\end{equation}
Consequently, if $f \in L^1(G)$ is nonnegative, then
\[
\int_G f(x) \, d\mu(x) \geq \frac{1}{m} \int_G f( \phi(x) ) \, d\mu(x).
\]
\end{lemma}

\begin{proof}
First, since $\phi$ is continuous and $G$ is compact, $\phi(G)$ is a compact subgroup of $G$. Since $G$ is Hausdorff, $\phi(G)$ is closed. In other words, $\phi(G)$ is a closed subgroup of $G$. 


Observe that since $G$ is partitioned into $m$-many translates of $\phi(G)$, $\mu(\phi(G)) = 1/m$.
Define $\lambda(B) = m \mu(B)$ for any Borel set $B \subseteq \phi(G)$. Now $\lambda$ is a translation-invariant probability measure on $\phi(G)$, and so it is equal to the normalized Haar measure on $\phi(G)$. By \cref{lem:pushforward}, $\lambda = \phi_* \mu$. This means that for any Borel set $B \subset \phi(G)$, we have $\mu(\phi^{-1}(B)) = m \mu(B)$.

Let $A$ be any Borel set in $G$. Since $A \subset \phi^{-1}(\phi(A))$, we have $\mu(A) \leq \mu(\phi^{-1}(\phi(A)) = m \mu(\phi(A))$, and the first assertion is proved. Applying the first assertion to the set $\phi^{-1}(A)$, we get the second assertion. 

The third assertion follows from the second one, and the fact that $f$ can be approximated by functions of the form $\sum_{i=1}^n c_i 1_{A_i}$ for Borel sets $A_i$ and $c_i \geq 0$.
\end{proof}

The next lemmas deal with images and preimages of Bohr sets under homomorphisms.
 
\begin{lemma} \label{lem:bohr-homo1}
Let $B \subset G$ be a Bohr-$(k, \eta)$ set and $\phi: G \rightarrow G$ be a continuous homomorphism. Then $\phi^{-1}(B)$ is also a Bohr-$(k, \eta)$ set.
\end{lemma}

\begin{proof}
If $B = \{ x \in G: |\gamma_i(x) -1| < \eta \textup{ for }i=1,  \ldots,  k\}$ is a Bohr-$(k, \eta)$ set, then $\phi^{-1}(B) = \{ x \in G: |\gamma_i \circ \phi (x) -1| < \eta \textup{ for }i=1,  \ldots,  k\}$ is also a Bohr-$(k, \eta)$-set.
\end{proof}


The next lemma is more surprising.

\begin{lemma}[cf. Griesmer {\cite[Lemma 1.7] {griesmer-br}}] \label{lem:bohr-homo2}
Let $B \subset G$ be a Bohr-$(k, \eta)$ set and $\phi: G \rightarrow G$ be a continuous homomorphism such that $[G:\phi(G)] = m < \infty$. Then $\phi(B)$ contains a Bohr-$(k', \eta')$ set, where $k', \eta'$ depend on $k, \eta$ and $m$.
\end{lemma}

\begin{proof}
Suppose $B = \{x \in G: |\gamma_i(x) - 1| < \eta \text{ for } 1 \leq i \leq k\}$ where $\gamma_i \in \Gamma$. Then 
\[
    A = \{x \in G: |\gamma_i(x) - 1| < \eta/4 \text{ for } 1 \leq i \leq k\}
\]
satisfies $A - A + A - A \subseteq B$. The bound \eqref{eq:mub} implies that $\mu(A) \geq (C_0\eta/4)^k$ for some absolute constant $C_0>0$. 

In view of \cref{lem:finite-index}, $\mu(\phi(A)) \geq \mu(A)/m \geq \frac{(C_0 \eta)^k}{4^k m}$. Therefore, by \cref{lem:bogolyubov_compact_group}, the set $\phi(B) \supseteq \phi(A) - \phi(A) + \phi(A) - \phi(A)$ is a Bohr-$(k', \eta')$ set where $k', \eta'$ depend only on $\mu(\phi(A))$, which is bounded below by $\frac{(C_0 \eta)^k}{4^k m}$.
\end{proof}

\subsection{Counting lemmas} 


\begin{lemma}[cf. {\cite[Lemma 2]{bourgain-roth}}]
\label{lem:bourgain_lemma_2}
 Let $\phi, \psi: G \rightarrow G$ be continuous homomorphisms such that $\phi(G), \psi(G)$ have finite index in $G$. Then for $f_1, f_2, f_3 \in L^\infty(G)$ and $K \in L^1(G)$ such that $\widehat{K} \in L^1(\Gamma)$, we have
\begin{equation} \label{eq:k}
\left| \iint_{G^2} f_1(x)f_2(x+ \phi(y))f_3(x+ \psi(y)) K(y) \, d\mu(x) d\mu(y) \right| \ll \| \widehat{f_1} \|_{\infty} \|f_2 \|_2 \| f_3\|_2 \| \widehat{K} \|_1  
\end{equation}
where the implied constant depends only on the indexes of $\phi(G)$ and $\psi(G)$ in $G$.
\end{lemma}

\begin{proof}
Since linear combinations of characters are dense in $L^1(G)$, without loss of generality, we can assume $f_1, f_2, f_3$ and $K$ are equal to their Fourier series. For $x \in G$, write $g(x) = \int_{G} f_2(x+ \phi(y) )f_3(x+ \psi(y) ) K(y) \, d\mu(y)$.

By Parseval's theorem,
\begin{equation*} 
\left| \int_G f_1(x) g(x) \, d\mu(x) \right|= \left| \sum_{\gamma \in \Gamma} \widehat{f_1}(\gamma) \widehat{g} (\overline{\gamma}) \right| \leq \| \widehat{f_1} \|_{\infty} \cdot \| \widehat{g} \|_1.
\end{equation*} 
Thus
\begin{eqnarray*}
g(x) &=& \int_{G} f_2(x+ \phi(y)) f_3(x+ \psi(y)) K(y) \, d\mu(y) \nonumber \\
&=& \int_G \left( \sum_{\gamma_2, \gamma_3, \gamma_0 \in \Gamma} \widehat{f_2}(\gamma_2) \gamma_2(x+ \phi(y)) \widehat{f_3}(\gamma_3) \gamma_3(x+ \psi(y)) \widehat{K}(\gamma_0) \gamma_0(y) \right) \, d\mu(y) \nonumber \\
&=& \int_G \left( \sum_{\gamma_2, \gamma_3, \gamma_0 \in \Gamma} \widehat{f_2}(\gamma_2) \widehat{f_3}(\gamma_3) \widehat{K}(\gamma_0) (\gamma_2 + \gamma_3)(x) (\gamma_2 \circ \phi +  \gamma_3 \circ \psi + \gamma_0 ) (y)\right) \, d\mu(y) \nonumber \\
&=&  \sum_{\substack{ \gamma_2, \gamma_3, \gamma_0 \in \Gamma,\\ \gamma_2 \circ \phi + \gamma_3 \circ \psi + \gamma_0=0}} \widehat{f_2}(\gamma_2) \widehat{f_3}(\gamma_3) \widehat{K}(\gamma_0) (\gamma_2 + \gamma_3)(x) 
\end{eqnarray*}
Consequently,
\begin{equation*} 
\widehat{g}(\gamma) = \sum_{\substack{ \gamma_2, \gamma_3, \gamma_0 \in \Gamma,\\ \gamma_2 \circ \phi + \gamma_3 \circ \psi + \gamma_0=0, \\ \gamma_2 + \gamma_3 = \gamma}} 
\widehat{f_2}(\gamma_2) \widehat{f_3}(\gamma_3) \widehat{K}(\gamma_0) 
\end{equation*}
and
\begin{equation*} 
\lVert \widehat{g} \rVert_1 \leq \sum_{\substack{ \gamma_2, \gamma_3, \gamma_0 \in \Gamma,\\ \gamma_2 \circ \phi + \gamma_3 \circ \psi + \gamma_0=0}} |\widehat{f_2}(\gamma_2) | \cdot | \widehat{f_3}(\gamma_3) | \cdot 
| \widehat{K}(\gamma_0) |.
\end{equation*}
Therefore, it suffices to show that for each $\gamma_0 \in \Gamma$, we have
\[
\sum_{\substack{ \gamma_2, \gamma_3 \in \Gamma,\\ \gamma_2 \circ \phi + \gamma_3 \circ \psi + \gamma_0=0}} |\widehat{f_2}(\gamma_2) | \cdot | \widehat{f_3}(\gamma_3) | \ll \|f_2 \|_2 \cdot \| f_3 \|_2,
\]
where the implicit constant depends only on the indexes $[G:\phi(G)]$ and $[G:\psi(G)]$.
By the Cauchy-Schwarz inequality and the Parseval's theorem, the left hand side is at most
\begin{eqnarray}
& & \| f_2 \|_2 \cdot \left( \sum_{\gamma_2} \left( \sum_{\substack{\gamma_3 \\ \gamma_3 \circ \psi = - \gamma_0 - \gamma_2 \circ \phi}} | \widehat{f_3}(\gamma_3) | \right)^2  \right)^{1/2} \nonumber \\
& \ll & \| f_2 \|_2 \cdot \left( \sum_{\gamma_2} \sum_{\substack{\gamma_3 \\ \gamma_3 \circ \psi = - \gamma_0 - \gamma_2 \circ \phi}} | \widehat{f_3}(\gamma_3) |^2  \right)^{1/2} \label{eq:index1} \\
& \ll & \| f_2 \|_2 \cdot \left( \sum_{\gamma_3}  | \widehat{f_3}(\gamma_3) |^2  \right)^{1/2} \label{eq:index2} \\
&=& \| f_2 \|_2 \cdot \| f_3 \|_2. \nonumber
\end{eqnarray}
In \eqref{eq:index1}, we use the fact that for each $\xi \in \Gamma$, there are at most $[G: \psi(G)]$ values of $\gamma_3$ such that $\gamma_3 \circ \psi = \xi$. Likewise, in \eqref{eq:index2}, we use the fact that for each $\xi \in \Gamma$, there are at most $[G: \phi(G)]$ values of $\gamma_2$ such that $\gamma_2 \circ \phi = \xi$. Both of these facts follow from \cref{lem:homo-solutions}.
\end{proof}

\begin{remark} Lemma \ref{lem:bourgain_lemma_2} is not true without the finite index assumption.
As a counterexample, we let $\phi(x) = x, \psi(x) = 2x$ and $G=\F_2 ^k$ for some large $k$. Let $n=|G| = 2^k$ and $\{\gamma_i: i \in [n]\}$ be the set of characters of $G$. Let $a_1, \ldots, a_n$ and $b_1, \ldots, b_n$ be nonnegative real numbers. The exact values of $a_i, b_i$ will be chosen later. For each $i \in [n]$, define 
\begin{itemize}
\item $\widehat{f_1}(\gamma_i) = \widehat{f_3}(\gamma_i) =1$,
\item $\widehat{f_2}(\gamma_i) = a_i$,
\item $\widehat{K}(\gamma_i) = b_i$.
\end{itemize}
It follows that $f_1(x) = f_3(x) = n \cdot 1_{x=0}$.
Then \eqref{eq:k} says that
\[
n(a_1 b_1 + \cdots + a_n b_n)^2 \ll (a_1^2 + \cdots + a_n^2) (b_1 + \cdots + b_n)^2.
\]
This is false by taking $a_1=b_1=1$ and $a_i = b_i =0$ for $i \neq 1$.
\end{remark}

While the previous lemma involves the configuration $x, x + \phi(y), x + \psi(y)$, the next one is concerned with $x, x + \phi(y)$ and $\psi(y)$. Its proof is almost identical and so we only highlight the differences. 

\begin{lemma} \label{lem:counting2}
 Let $\phi, \psi: G \rightarrow G$ be continuous homomorphisms such that $\phi(G), \psi(G)$ have finite index in $G$. Then for $f_1, f_2, f_3 \in L^\infty(G)$, we have
\begin{equation} 
\left| \iint_{G^2} f_1(x) f_2(x+ \phi(y))f_3(\psi(y)) \, d\mu(x) d\mu(y) \right| \ll \| \widehat{f_1} \|_{\infty} \|f_2 \|_2 \| f_3\|_2 
\end{equation}
where the implicit constant depends only on the indexes of $\phi(G)$ and $\psi(G)$ in $G$.
\end{lemma}
\begin{proof}

Similar to the proof of \cref{lem:bourgain_lemma_2}, without loss of generality, we can assume $f_1, f_2, f_3$ are equal to their Fourier series. For $x \in G$, write $g(x) = \int_{G} f_2(x+ \phi(y) )f_3(\psi(y) ) \, d\mu(y)$ and then by Parseval's theorem,
\begin{equation*} 
\left| \int_G f_1(x) g(x) \, d \mu(x) \right|= \left| \sum_{\gamma \in \Gamma} \widehat{f_1}(\gamma) \widehat{g} (\overline{\gamma}) \right| \leq \| \widehat{f_1} \|_{\infty} \cdot \| \widehat{g} \|_{1}.
\end{equation*} 
Moreover, we also have
\begin{eqnarray*}
g(x) 
&=& \int_G \left( \sum_{\gamma_2, \gamma_3 \in \Gamma} \widehat{f_2}(\gamma_2) \gamma_2(x+ \phi(y)) \widehat{f_3}(\gamma_3) \gamma_3(\psi(y)) \right) \, d\mu(y) \nonumber \\
&=&  \sum_{\substack{ \gamma_2, \gamma_3 \in \Gamma,\\ \gamma_2 \circ \phi + \gamma_3 \circ \psi = 0}} \widehat{f_2}(\gamma_2) \widehat{f_3}(\gamma_3) \gamma_2(x). 
\end{eqnarray*}
As a consequence,
\begin{equation*} 
\widehat{g}(\gamma) = \sum_{\substack{ \gamma_2, \gamma_3 \in \Gamma,\\ \gamma_2 \circ \phi + \gamma_3 \circ \psi = 0, \\ \gamma_2 = \gamma}} 
\widehat{f_2}(\gamma_2) \widehat{f_3}(\gamma_3) =
\widehat{f_2}(\gamma) \sum_{\substack{\gamma_3 \in \Gamma, \\ \gamma \circ \phi + \gamma_3 \circ \psi = 0}}  \widehat{f_3}(\gamma_3)
\end{equation*}
and so
\begin{equation*} 
\| \widehat{g} \|_{1} \leq \sum_{\substack{\gamma_2, \gamma_3 \in \Gamma, \\ \gamma_2 \circ \phi + \gamma_3 \circ \psi = 0}} \left| \widehat{f_2}(\gamma_2) \right| \left| \widehat{f_3}(\gamma_3) \right|.
\end{equation*}
On the other hand, we have
\begin{eqnarray}
\sum_{\substack{\gamma_2, \gamma_3 \in \Gamma, \\ \gamma_2 \circ \phi + \gamma_3 \circ \psi = 0}} \left| \widehat{f_2}(\gamma_2) \right| \left| \widehat{f_3}(\gamma_3) \right| &=& \sum_{\gamma_2 \in \Gamma} \left( \left| \widehat{f_2}(\gamma_2) \right| \sum_{\gamma_3 \in \Gamma, \atop{\gamma_3 \circ \psi = - \gamma_2 \circ \phi}} \left| \widehat{f_3}(\gamma_3) \right|\right) \nonumber\\
&\leq& \left( \sum_{\gamma_2 \in \Gamma} \left| \widehat{f_2}(\gamma_2) \right|^2 \right)^{1/2} \left( \sum_{\gamma_2 \in \Gamma} \left( \sum_{\gamma_3 \in \Gamma, \atop{\gamma_3 \circ \psi = - \gamma_2 \circ \phi}} \left| \widehat{f_3}(\gamma_3) \right| \right)^2 \right)^{1/2} \nonumber \\
&\ll& \lVert f_2 \rVert_2 \left( \sum_{\gamma_2 \in \Gamma} \sum_{\gamma_3 \in \Gamma,\atop{\gamma_3 \circ \psi = - \gamma_2 \circ \phi}} \left|  \widehat{f_3}(\gamma_3) \right|^2 \right)^{1/2} \label{eq:new_index1} \\
&\ll& \lVert f_2 \rVert_2 \left( \sum_{\gamma_3 \in \Gamma} \left| \widehat{f_3}(\gamma_3) \right|^2 \right)^{1/2} \label{eq:new_index2} \\
&=& \lVert f_2 \rVert_2 \lVert f_3 \rVert_2. \nonumber
\end{eqnarray}
In \eqref{eq:new_index1}, we use the fact that for each $\xi \in \Gamma$, there are $\leq [G: \psi(G)]$ values of $\gamma_3$ such that $\gamma_3 \circ \psi = \xi$ while \eqref{eq:new_index2} follows from the fact that there are $\leq [G: \phi(G)]$ values of $\gamma_2$ such that $\gamma_2 \circ \phi = \xi$.
\end{proof}

\section{Bohr sets and partitions}
\label{sec:partition}

\subsection{Monochromatic configurations}
We make some preparations before the proof of \cref{th:counting-partition}. In this section, we only need $G$ to be a commutative semigroup with neutral element. Fix $k+1$ commuting (semigroup)  homomorphisms $\psi, \phi_1, \ldots, \phi_k: G \rightarrow G$ with $\psi \neq 0$. 
We write 
\[
\Phi_m = \{ \psi^{i_0} \circ \phi_1^{i_1} \circ \cdots \circ \phi_k^{i_k} : 0 \leq i_0, i_1, \ldots, i_k \leq m \} \cup \{ 0 \}
\] (where $\phi^i$ is the $i$-th composition of $\phi$).

For \textit{formal variables} $x_1, \ldots, x_n$, we write
\[
S_m(x_1,\ldots, x_n) = \left\{ \sum_{i=1}^n \xi_i (x_i) : \xi_i \in \Phi_m \right\} 
\]
and we refer to $S_m(x_1,\ldots, x_n)$ as the $S_{m,n}$-\textit{set} with \textit{generators} $x_1, \ldots, x_n$. There is a canonical bijection between $S_m(x_1, \ldots, x_n)$ and $\Phi_m^n$ defined by
\[
    \sum_{i=1}^n \xi_i(x_i) \mapsto (\xi_1, \ldots, \xi_n).
\]
 For an element $x = \sum_{i=1}^n \xi_i (x_i) \in S_m(x_1,\ldots, x_n)$, by the \textit{support} of $x$ we mean the set $\{ i \in [n]: \xi_i \neq 0\}$.
The goal of this section is to prove the following:

\begin{theorem} \label{th:mpc-hom} 
For any $r > 0$, there exist $n$ and $m$ such that under any $r$-coloring of $S_m(x_1,\ldots, x_n)$, there is a monochromatic configuration
\[
\{ \psi(y), x, x + \phi_1(y), \ldots, x+ \phi_k(y) \},
\]
where $x, y$ have nonempty and disjoint supports.
\end{theorem}
The fact that the supports of $x$ and $y$ are nonempty and disjoint will be crucial in our applications (\cref{th:counting-partition} and \cref{prop:counting-brauer}). \cref{th:mpc-hom} follows from \cref{prop:mpc} below whose proof requires the multidimensional Hales-Jewett theorem (for a reference, see \cite[Theorem 7, p.40]{ramsey-book}). We recall the theorem here for reader's convenience.


The set $[t]^N = \{(x_1, \ldots, x_N): x_i \in [t] \}$ is called a \textit{cube} of dimension $N$ over $t$ elements. Let $[N] = A_0 \cup A_1 \cup \cdots \cup A_m$ be any disjoint partition of $[N]$, where $A_i \neq \varnothing$ for $i \neq 0$ ($A_0$ may be empty), and $f: A_0 \rightarrow [t]$ be any map. Define a map $g: [t]^m \rightarrow [t]^N$ by assigning to each $(y_1, \ldots, y_m) \in [t]^m$ the element $(x_1, \ldots, x_N) \in [t]^N$, where
\begin{equation}
x_i = 
\begin{cases}
f(i), \quad &\textup{if}\ i \in A_0 \\
y_j, \quad &\textup{if}\ i \in A_j \text{ for }j \in [m].
\end{cases}
\end{equation}
A \textit{combinatorial space} of dimension $m$ is the image of $g$ for some choice of $A_0, A_1, \ldots, A_m$ and $f$. We can now state:

\begin{theorem}[Multidimensional Hales-Jewett] For any $r, t, m$, there exists a number $N=HJ(t, m;r)$ such that whenever $[t]^N$ is $r$-colored, there must be a monochromatic combinatorial space of dimension $m$.
\end{theorem}

Using this, we can prove the following proposition. (Recall that $k$ is fixed from the beginning of this section.)

\begin{proposition} \label{prop:mpc}
 For any $r>0$ and $\ell >0$, there exist $n=n(k,\ell, r)$ and $m=m(k,\ell, r)$ such that under any $r$-coloring of $S_{m}(x_1,\ldots, x_n)$, there are elements 
$y_1, \ldots, y_\ell \in S_m(x_1,\ldots, x_n)$ with nonempty and disjoint supports, such that for each $i \in [\ell]$, the elements
\[
\psi(y_i) + \sum_{1 \leq j \leq i-1} \xi_j (y_j) \qquad \textup{ where } \xi_j \in \{ 0, \psi, \phi_1, \ldots, \phi_k \}
\]
have the same color (i.e. their color depends only on $i$).
\end{proposition}

\begin{proof}
The number of colors $r$ will be fixed throughout. We will proceed by induction on $\ell$. When $\ell=1$ the statement is obvious. Suppose the statement is true for $\ell$, we will prove it is true for $\ell+1$. 

Write $n' = n(k,\ell, r), m'=m(k,\ell, r)$. We define $m = m(k,\ell + 1, r):=|\Phi_{m'+1}| +1$, $N:=HJ( |\Phi_{m'+1}|, n';r)$ and $n = n(k, \ell+1,r): = 1 + N$.

Consider an arbitrary $r$-coloring of $S_{m}(x_1,\ldots, x_n)$. An $r$-coloring of $S_{m}(x_1,\ldots, x_n)$ induces an $r$-coloring of $\Phi_{m'+1}^{N}$ by assigning to 
$(a_1, \ldots, a_{N}) \in (\Phi_{m'+1})^{N}$ the color of 
\[
\psi (x_n) + \sum_{i=1}^{N} \psi \circ a_i (x_i).
\]
Since $N=HJ( |\Phi_{m'+1}|, n';r)$, there are a disjoint partition
\[
[N] = A_0 \cup A_1 \cup \cdots \cup A_{n'}, \quad A_i \neq \varnothing, \forall i \neq 0
\]
and functions $f_i \in \Phi_{m' + 1}$ for $i \in A_0$ such that when $\zeta_1, \zeta_2, \ldots, \zeta_{n'}$ range over $\Phi_{m'+1}$, all the elements 
\[
\psi(x_n) + \sum_{i=1}^{N} \psi \circ a_i (x_i),
\] 
with
\begin{equation*}
a_i = 
\begin{cases}
f_i, \quad &\textup{if}\ i \in A_0 \\
\zeta_j, \quad &\textup{if}\ i \in A_j \text{ for } 1 \leq j \leq n'
\end{cases}
\end{equation*}
have the same color.

Write $z_j = \sum_{i \in A_j} \psi(x_i)$ for $1 \leq j \leq n'$, and $z_{n'+1} = x_{n} + \sum_{i \in A_0} f_i(x_i)$. Then all the $z_j$ have nonempty and disjoint supports, and all elements of the form
\[
\psi(z_{n'+1}) + \sum_{j=1}^{n'} \zeta_j (z_j), \qquad \zeta_j \in \Phi_{m'+1},
\]
have the same color.

By the inductive hypothesis, there exists a sequence $y_1, \ldots, y_\ell \in S_{m'} (z_1, \ldots, z_{n'})$ having nonempty and disjoint supports such that for each $i=1, \ldots, \ell$, the elements
\[
\psi(y_i) + \sum_{1 \leq j \leq i-1} \xi_j (y_j) \qquad \textup{ where } \xi_j \in \{ 0, \psi, \phi_1, \ldots, \phi_k \}
\]
have the same color. We now set $y_{\ell+1} = z_{n'+1}$. Clearly the elements 
\[
\psi(y_{\ell+1}) + \sum_{1 \leq j \leq \ell} \xi_j (y_j) \qquad \textup{ where } \xi_j \in \{ 0, \psi, \phi_1, \ldots, \phi_k \}
\]
are of the form
\[
\psi(z_{n'+1}) + \sum_{j=1}^{n'} \zeta_j (z_j), \qquad \zeta_j \in \Phi_{m'+1},
\]
and so they have the same color. Thus  Proposition \ref{prop:mpc} is proved.
\end{proof}

\begin{proof}[Proof of \cref{th:mpc-hom}]
Applying Proposition \ref{prop:mpc} with $\ell=r+1$, we can find a sequence $y_1, \ldots, y_{r+1}$ satisfying the conclusion of that proposition. Let $c(i)$ be the color of
\[
\psi(y_i) + \sum_{1 \leq j \leq i-1} \xi_j (y_j) \qquad \textup{ where } \xi_j \in \{ 0, \psi, \phi_1, \ldots, \phi_k \}.
\]
Then there exist $1 \leq u < v \leq r+1$ such that $c(u) = c(v)$. Hence the elements
\[
\psi(y_u), \psi(y_v), \psi(y_v) + \phi_1(y_u), \ldots, \psi(y_v) + \phi_k(y_u), 
\]
have the same color, and we are done (with $x=\psi(y_v), y=y_u$).
\end{proof}

\subsection{Proofs of \texorpdfstring{\cref{th:counting-partition}}{Theorem 1.11} and \texorpdfstring{\cref{th:main-partition}}{Theorem 1.5}}

Using Theorem \ref{th:mpc-hom} we can now prove \cref{th:counting-partition}, which we recall for convenience:

\begin{theorem*} 
Suppose $\psi, \phi_1, \ldots, \phi_k: G \to G$ are continuous homomorphisms satisfying:
\begin{enumerate}
\item $\psi, \phi_1, \ldots, \phi_k$ are commuting, and

\item $\psi(G), \phi_1(G), \ldots, \phi_k(G)$ have finite index in $G$.
\end{enumerate}
Suppose $G = \bigcup_{i=1}^r A_i$ is a partition of $G$ into measurable sets. Then 
\begin{equation*} 
\sum_{i=1}^r \iint_{G^2} 1_{A_i}(\psi(t)) 1_{A_i}(x) 1_{A_i}(x + \phi_1(t)) \cdots 1_{A_i}(x + \phi_k(t)) \, d \mu(x) d\mu(t) \geq c_3
\end{equation*} 
for some positive constant $c_3$ depending only on $r, k$ and the aforementioned indexes.
\end{theorem*}

\begin{proof}Consider the set $S_m(x_1, \ldots, x_n)$ given by Theorem \ref{th:mpc-hom}, where we now let $x_1, \ldots, x_n$ vary over $G$. Note that for any non-zero $\phi \in \Phi_m$, we have $[G:\phi(G)] < \infty$ by \cref{lem:composition}. Let $R=R(x_1, \ldots, x_n)$ be the set of all pairs $(z, t)$ where $z, t \in S_m(x_1, \ldots, x_n)$ have nonempty and disjoint supports.

Suppose $G = \bigcup_{i=1}^r A_i$. For $i \in [r]$, we define
\[
T_i : = \iint_{G^2} 1_{A_i}( \psi(y) )  1_{A_i}(x) 1_{A_i}(x+ \phi_1(y)) \cdots 1_{A_i}(x+ \phi_k(y)) \, d\mu(x) d\mu(y). 
\]
Let $(z,t) \in R$ be arbitrary, and suppose
\[
z = \sum_{u \in U} \zeta_u(x_u) \qquad \textup{and} \qquad t = \sum_{v \in V} \xi_v (x_v)
\]
where $U, V \subset [n]$ are nonempty and disjoint and $\zeta_u, \xi_v \in \Phi_m \setminus \{ 0\}$. We have
\begin{eqnarray*}
& & \int_{G^n}  1_{A_i}( \psi(t)) 1_{A_i}(z) 
1_{A_i}(z+ \phi_1(t)) \cdots 1_{A_i}(z+ \phi_k(t)) \, d\mu(x_1) \cdots d\mu(x_n) \\
& = & \int_{G^n} 1_{A_i} \left( \sum_{v \in V} \psi( \xi_v (x_v) ) \right) 
1_{A_i} \left( \sum_{u \in U} \zeta_u(x_u)\right)
1_{A_i} \left( \sum_{u \in U} \zeta_u(x_u)  + \phi_1 \left( \sum_{v \in V} \xi_v (x_v) \right) \right) \cdots \\
& & \qquad 1_{A_i} \left( \sum_{u \in U} \zeta_u(x_u)  + \phi_k \left( \sum_{v \in V} \xi_v (x_v) \right)  \right) \, d\mu(x_1) \cdots d \mu(x_n) \\
& \ll & \int_{G^n} 1_{A_i} \left( \sum_{v \in V} \psi( x_v ) \right) 
1_{A_i} \left( \sum_{u \in U} x_u \right)
1_{A_i} \left( \sum_{u \in U} x_u  + \phi_1 \left( \sum_{v \in V} x_v \right) \right) \cdots \\
& & \qquad 1_{A_i} \left( \sum_{u \in U} x_u  + \phi_k \left( \sum_{v \in V} x_v \right)  \right) \, d\mu(x_1) \cdots d \mu(x_n) \\
& = & \iint_{G^2} 1_{A_i}( \psi(y) ) 1_{A_i}(x) 1_{A_i}(x  + \phi_1(y) ) \cdots 1_{A_i}(x  + \phi_k(y) ) \, d\mu(x) \, d\mu(y) \\
&=& T_i, 
\end{eqnarray*}
by $|U|+|V|$ applications of \cref{lem:finite-index}.

Fixing any $(x_1, \ldots, x_n) \in G^n$, the set $S_m(x_1, \ldots, x_n)$ becomes a subset of $G$ and so the coloring $G = \bigcup_{i=1}^r A_i$ naturally induces a coloring of $S_m(x_1, \ldots, x_n)$. By \cref{th:mpc-hom}, there exists $i \in [r]$ and $(z, t) \in R$ such that
\[
    \psi(t), z, z+ \phi_1(t), \cdots, z+ \phi_k(t) \in A_i.
\]
Thus, for every $(x_1, \ldots, x_n) \in G^n$,
\[
    \sum_{i=1}^r \sum_{(z, t) \in R } 1_{A_i}( \psi(t)) 1_{A_i}(z) 1_{A_i}(z+ \phi_1(t)) \cdots 1_{A_i}(z+ \phi_k(t)) \geq 1.
\]
It follows that
\begin{eqnarray*}
1 & \leq & \int_{G^n} \sum_{i=1}^r \sum_{(z, t) \in R } 1_{A_i}( \psi(t)) 1_{A_i}(z) 1_{A_i}(z+ \phi_1(t)) \cdots 1_{A_i}(z+ \phi_k(t)) \, d\mu(x_1) \cdots d\mu(x_n) \\
&\leq& \sum_{i=1}^r  \int_{G^n} \sum_{(z, t) \in R } 1_{A_i}( \psi(t)) 1_{A_i}(z) 1_{A_i}(z+ \phi_1(t)) \cdots 1_{A_i}(z+ \phi_k(t)) \, d\mu(x_1) \cdots d\mu(x_n)  \\
& \ll & \sum_{i=1}^r T_i,   
\end{eqnarray*}
thus finishing the proof.
\end{proof}

To prove \cref{th:main-partition}, we will need the following proposition. 
With an eye to potential applications, we state and prove a slightly stronger version than what is needed.



\begin{proposition} \label{prop:partition-bohr-stronger} 
Let $\phi, \psi: G \to G$ be commuting continuous homomorphisms with images of finite index. Suppose $f_1, \ldots, f_r: G \to [0,1]$ are measurable functions such that $\sum_{i=1}^r f_i \geq 1$ pointwise. For $w \in G$, define
\[
  R_i(w) =  \iint_{G^2}  f_i(\psi(y)) f_i(x+w) f_i(x + \phi(y)) \ d\mu(x) d\mu(y). 
\]
Then there are $c, k, \eta >0$ depending only on $r$ and the indexes above such that for some $i \in [r]$, the set
 $\{ w \in G: R_i(w) > c\}$ contains a Bohr-$(k,\eta)$ set. 
\end{proposition}
\begin{proof}

For $i \in [r]$, let $A_i = \{x \in G: f_i(x) \geq 1/r\}$. Since $\sum_{i=1}^r f_i \geq 1$ pointwise, $G = \bigcup_{i=1}^r A_i$. In light of \cref{th:counting-partition}, there exists a constant $c$ depending only on $r$ and the indexes and an $i \in [r]$ such that 
\[
    \iint_{G^2} 1_{A_i} (\psi(y)) 1_{A_i}(x) 1_{A_i}(x + \phi(y)) \ d\mu(x) d\mu(y) > c.
\]
It then follows that
\[
    R_i(0) \geq \frac{c}{r^3}.  
\]
On the other hand, by \cref{lem:counting2}, for every $w \in G$,
\begin{equation*}
\label{eq:partition_R_i}
    |R_i(w) - R_i(0)| \ll \lVert \widehat{f_i} - \widehat{f_{i,w}} \rVert_{\infty},
\end{equation*}
where the implicit constant depends only on the indexes of $\phi(G)$ and $\psi(G)$ in $G$. 
Hence, there exists a constant $c'$ such that $R_{i}(w) \geq \frac{c}{2r^3}$ if
\begin{equation*}
    \lVert \widehat{f_{i}} - \widehat{f_{i,w}} \rVert_{\infty} < c'.
\end{equation*}
By \cref{lem:translate}, the set of such $w$ contains a Bohr-$(k,\eta)$ set, where $k$ and $\eta$ depend only on $c'$.
\end{proof}
 
\cref{th:main-partition} is now a special case of the next theorem with $\psi_1 = \phi_2$ and $\psi_2 = \phi_1$. 

\begin{theorem}\label{th:main-partition-true}
Let $G=\bigcup_{i=1}^r A_i$ be a partition into measurable sets. Let $\phi_1, \phi_2, \psi_1, \psi_2: G \rightarrow G$ be continuous homomorphisms satisfying the following:
\begin{enumerate}
\item $\phi_2 \circ \psi_2 = \phi_1 \circ \psi_1$,
\item $\psi_1 \circ \psi_2 = \psi_2 \circ \psi_1$,
\item $\phi_1(G), \psi_1(G), \psi_2(G)$ have finite index in $G$.
\end{enumerate}
Then for some $1 \leq i \leq r$, the set $\phi_1(A_i) - \phi_1(A_i) + \phi_2(A_i)$ contains a Bohr-$(k, \eta)$ set, where $k$ and $\eta$ depend only on $r$ and the indexes of $\phi_1(G), \psi_1(G), \psi_2(G)$ in $G$.
\end{theorem}

\begin{proof}
Suppose $G = \bigcup_{i=1}^r A_i$. We apply \cref{prop:partition-bohr-stronger} with $f_i = 1_{A_i}$ and $(\psi_1, \psi_2)$ in place of $(\psi, \phi)$. Then for some $i$, the set $\{ w \in G: R_i(w) > c\}$ contains a Bohr-$(k,\eta)$ set $B$. This means that for $w \in B$, there exist $x, y \in G$ such that
\[
x+w, \psi_2(y), x + \psi_1(y) \in A_i.
\]
Since
\[
\phi_1(x+w) + \phi_2( \psi_2(y)) - \phi_1( x + \psi_1(y)) = \phi_1(w),
\]
we conclude that $\phi_1(B) \subset \phi_1(A_i) + \phi_2(A_i) - \phi_1(A_i)$. Our theorem now follows from \cref{lem:bohr-homo2}.
\end{proof}

\begin{remark} \label{rk:auto2}
Here we explain why the commutativity condition of $\phi_1$ and $\phi_2$ in \cref{th:main-partition} is not necessary if $\phi_1$ or $\phi_2$ is an automorphism (see \cref{remark:first_one}). If $\phi_1$ is an automorphism, in \cref{th:main-partition-true}, we take $\psi_1 = \phi_1^{-1} \circ \phi_2$ and $\psi_2 = \textup{Id}$, the identity homomorphism. Then the first two conditions of \cref{th:main-partition-true} are satisfied. As for the third condition, we have $\psi_1(G) = \phi_1^{-1} \circ \phi_2(G),$ which has finite index in $G$ by \cref{lem:composition}. A similar argument applies to the case $\phi_2$ is an automorphism.
\end{remark}


\section{Bohr sets and sets of positive measure}
\label{sec:density}

\subsection{A regularity lemma}
The goal of this section is to prove \cref{prop:regularity_lemma}. As mentioned in the introduction, this argument has its genesis in Bourgain \cite{bourgain-roth}. Bourgain's ideas were elaborated by Tao \cite{tao}, who proved Roth's theorem in compact abelian groups that are 2-divisible; and by Bergelson-Host-McCutcheon-Parreau \cite[Theorem 4.1]{bhmp}, who proved Roth's theorem for dilations on the torus $\R / \Z$. We streamline and generalize Bergelson-Host-McCutcheon-Parreau's argument to deal with homomorphisms on arbitrary compact abelian groups. This generalization requires non-trivial modifications; especially, we will make use of \cref{lem:translate} and \cref{lem:kernel_a}. 



\begin{lemma}[cf. {\cite[Lemma 4.2]{bhmp}}]
\label{lem:J}
Let $\phi, \psi: G \rightarrow G$ be continuous homomorphisms such that $\phi(G), \psi(G)$ and $(\phi-\psi)(G)$ have finite indexes in $G$. For $f \in L^\infty(G)$, define
\[
J(f) = \iint_{G^2} f(x)f(x+ \phi(y))f(x+ \psi(y)) \, d\mu(x) d\mu(y).
\]
Then for any measurable functions $f, g: G \to [0,1]$, 
    \begin{equation*} \label{eq:c}
        |J(f) - J(g)| \ll \| \widehat{f} - \widehat{g} \|_\infty,
    \end{equation*}
where the implicit constant depends only on the aforementioned indexes. 
\end{lemma}

\begin{proof}
We have
\begin{eqnarray*}
J(f) - J(g) &=& \iint_{G^2} (f-g)(x) \cdot f(x+ \phi(y)) \cdot f(x+ \psi(y)) \, d\mu(x) d\mu(y) \\
& & + \iint_{G^2} g(x) \cdot ( f-g) (x+ \phi(y)) \cdot f(x+ \psi(y)) \, d\mu(x) d\mu(y) \\
& & + \iint_{G^2} g(x) \cdot g(x+ \phi(y)) \cdot (f - g) (x + \psi(y)) \, d\mu(x) d\mu(y).
\end{eqnarray*}
The lemma now follows from Lemma \ref{lem:bourgain_lemma_2} and the assumptions on $\phi$ and $\psi$.
\end{proof}


\begin{proposition}[Regularity Lemma]
\label{prop:regularity_lemma} Let $f: G \to [0, 1]$ be a measurable function with $\int_G f \, d \mu = \delta >0$. Let $\phi, \psi: G \to G$ be continuous homomorphisms such that $\phi(G), \psi(G)$ and $(\phi- \psi)(G)$ have finite indexes in $G$. Then for every $\epsilon > 0$, there exist a constant $C$ that depends only on $\delta, \epsilon$ and the indexes above, a kernel $K: G \to \R_{\geq 0}$, and a decomposition $f = f_{st} + f_{er} + f_{un}$ such that
\begin{enumerate}
    \item $\lVert K \rVert_{\infty} < C$,
    \item $\lVert f_{st} \rVert_{\infty} \leq 1$, $\lVert f_{er} \rVert_{\infty} \leq 2$ and $\lVert f_{un} \rVert_{\infty} \leq 2$,
    \item $\lVert f_{er} \rVert_2 < \epsilon$,
    \item $\lVert \widehat{f}_{un} \rVert_{\infty} \lVert \widehat{K} \rVert_1 < \epsilon$,
    \item $J'(f_{st}) := \displaystyle \iint_{G^2} f_{st}(x) f_{st}(x+ \phi(t)) f_{st}(x+ \psi(t)) K(t) \, d\mu(x) d\mu(t) \geq \delta^3 - \epsilon$.
\end{enumerate}

\end{proposition}
\begin{proof}
For $t \in G$, let
		\[
			d(t) := \max \left( \| \widehat{f} - \widehat{f_{t}} \|_\infty, \| \widehat{f} - \widehat{f_{\phi(t)}} \|_\infty, \| \widehat{f} - \widehat{f_{\psi(t)}} \|_\infty  \right).
		\]
Fixing $\epsilon > 0$, we define sequences $\eta_n \in (0, 1/2]$, $\kappa_n \in (0, \infty)$ and finite sets $\Lambda_n \subseteq \Gamma$ recursively as follows: 

First set $\eta_0 = 1/2$. For $n \geq 0$, \cref{lem:translate} implies that there exists a set $\Lambda_n \in \Gamma$ with $|\Lambda_n| \leq 12/\eta_n^2$ such that $D(\eta_n) := \{t \in G: d(t) \leq \eta_n\}$ contains a Bohr set $B(\Lambda_n; \eta_n)$.  
For $\eta \in (0, 1/2]$, define $\nu(\eta) = (C_0 \eta)^{12/\eta^2}$ where $C_0$ is the constant found in \cref{lem:kernel_a}; in particular, $\nu(\eta_n) = (C_0 \eta_n)^{12/\eta_n^2} \leq (C_0 \eta_n)^{|\Lambda_n|}$. Put
\[
         \kappa_n = \nu(\eta_n)^{-1/2} \text{ and } \eta_{n+1} = \min \left \{\eta_n, \frac{\epsilon^2}{4 \kappa_n^2}, \epsilon \nu \left( \frac{\epsilon}{2 \kappa_n} \right) \right\}.
\]

In view of \cref{lem:kernel_a}, for $n \geq 0$, there is a kernel $K_n: G \to [0, \infty)$ such that
    \[
        \widehat{K}_n \geq 0,  \| \widehat{K_n}\|_1 = \lVert K_n \rVert_{\infty} \leq 1/\nu(\eta_n)
    \]
    and $K_n$ is supported on $B(\Lambda_n; \eta_n) \subseteq D(\eta_n)$. 
    We define
    \[
        f_n = f * K_n.
    \]
	\noindent \textbf{Claim 1:} 
	\[
	\| \widehat{f} - \widehat{f_n} \|_{\infty} = \sup_{\gamma \in \Gamma} \left| \widehat{f}(\gamma) (1-\widehat{K_n}(\gamma)) \right| \leq \eta_n.
	\]
	Indeed, by construction, $K_n$ is supported on $D(\eta_n)$, and every $t \in D(\eta_n)$ satisfies $\left| \widehat{f}(\gamma) (1 - \gamma(t)) \right| \leq \eta_n$ for all $\gamma \in \Gamma$. Therefore, for all $\gamma \in \Gamma$,
	\begin{eqnarray*}
	\left| \widehat{f}(\gamma) \right| \left| 1-\widehat{K_n}(\gamma) \right| 
	&\leq& 
	\left| \widehat{f}(\gamma) \right| \int_{G} K_n(x) \left| 1-\overline{\gamma(x)} \right| \, d\mu(x) \\
	&=& \left| \widehat{f}(\gamma) \right| \int_{D(\eta_n)} K_n(x) \left| 1-\overline{\gamma(x)} \right| \, d\mu(x) \\
	&=& \int_{D(\eta_n)} K_n(x) \left| \widehat{f}(\gamma) \right| \left| 1-\overline{\gamma(x)} \right| \, d\mu(x)\\
	&\leq& \eta_n \int_{D(\eta_n)} K_n(x) \, d\mu(x) \leq \eta_n.
	\end{eqnarray*}
	
\noindent	\textbf{Claim 2:} 
    \[
        \lVert f_{n+1} - f_n \rVert_2^2 \leq \lVert f_{n+1} \rVert_2^2 - \lVert f_n \rVert_2^2 + 2 \eta_{n+1} \kappa_n^2.
    \]
	Indeed, we have
	\begin{eqnarray*}
	\lVert f_{n+1} - f_n \rVert_2^2 
&= &	\lVert \widehat{f_{n+1}} - \widehat{f_n} \rVert_2^2 \\
&= &	\lVert \widehat{f_{n+1}} \rVert_2^2 + \lVert \widehat{f_n} \rVert_2^2 - \sum_{\gamma \in \Gamma} 
\left(\widehat{f_{n+1}}(\gamma) \overline{\widehat{f_{n}}(\gamma)} + \overline{\widehat{f_{n+1}}(\gamma)} \widehat{f_{n}}(\gamma) \right)\\
&= &	\lVert \widehat{f_{n+1}} \rVert_2^2 - \lVert \widehat{f_n} \rVert_2^2 + 2 \lVert \widehat{f_n} \rVert_2^2 - \sum_{\gamma \in \Gamma}	\left| \widehat{f}(\gamma) \right|_2^2 \widehat{K_n}(\gamma) \widehat{K_{n+1}}(\gamma)\\
&= & \lVert \widehat{f_{n+1}} \rVert_2^2 - \lVert \widehat{f_n} \rVert_2^2 + 2 \sum_{\gamma \in \Gamma}	\left| \widehat{f}(\gamma) \right|_2^2 \widehat{K_n}(\gamma) 
\left( \widehat{K_n}(\gamma) - \widehat{K_{n+1}}(\gamma) \right)\\
& \leq & \lVert \widehat{f_{n+1}} \rVert_2^2 - \lVert \widehat{f_n} \rVert_2^2 + 2 \sum_{\gamma \in \Gamma}	\left| \widehat{f}(\gamma) \right|^2 \widehat{K_n}(\gamma) 
\left(1 - \widehat{K_{n+1}}(\gamma) \right) \\
& \leq & \lVert \widehat{f_{n+1}} \rVert_2^2 - \lVert \widehat{f_n} \rVert_2^2 + 2 \sup_{\gamma \in \Gamma} |\widehat{f}(\gamma)| \left( 1 - \widehat{K_{n+1}}(\gamma) \right) \cdot \| \widehat{K_n} \|_1 \\
& \leq & \lVert \widehat{f_{n+1}} \rVert_2^2 - \lVert \widehat{f_n} \rVert_2^2 + 2 \eta_{n+1} \kappa_n^2
	\end{eqnarray*}
 and the claim is proved.
   
		Since $\eta_{n+1} \leq \epsilon^2/(4 \kappa_n^2)$, we have
    \[
        \lVert f_{n+1} - f_n \rVert_2^2 \leq \lVert f_{n+1} \rVert_2^2 - \lVert f_n \rVert_2^2 + \epsilon^2/2.
    \]
     Let $M$ be the smallest integer such that $M \geq 2/\epsilon^2$. Then because $\lVert f_M \rVert_2^2 \leq \lVert f_M \rVert_{\infty}^2 \leq 1$,
    \[
        \sum_{n=0}^{M-1} \lVert f_{n+1} - f_n \rVert_2^2 \leq \lVert f_M \rVert_2^2 - \lVert f_0 \rVert_2^2 + M \epsilon^2/2 \leq 1 + M \epsilon^2/2 \leq M \epsilon^2.
    \]
    Therefore there exists $0 \leq n \leq M - 1$ such that
    \[
        \lVert f_{n+1} - f_n \rVert_2 \leq \epsilon.
    \]
    From now on, we fix this $n$.
    Next consider the expression
    \[
        I_n(t) = \int_G f_n(x) f_n(x+\phi(t)) f_n(x + \psi(t)) \, d \mu(x) \; \; \text{ for } t \in G.
    \]
    By the same algebra as in the proof of \cref{lem:J},
    \[
        |I_n(0) - I_n(t)| \leq \| f_n - (f_n)_{\phi(t)} \|_1 + \| f_n - (f_n)_{\psi(t)} \|_1. 			
    \]
	Note that
	\begin{eqnarray*}
		\| f_n - (f_n)_{\phi(t)} \|^2_1 &\leq & \| f_n - (f_n)_{\phi(t)} \|_2^2 = \| (f - f_{\phi(t)})*K_n \|_2^2 \\
		&=& \sum_{\gamma \in \Gamma} \left| \widehat{K_n}(\gamma) \right|^2 \left|\widehat{f}(\gamma) - \widehat{f_{\phi(t)}}(\gamma) \right|^2 \\
		&\leq & \| \widehat{K_n} \|_1 d(t)^2 \leq \kappa_n^2 d(t)^2.
		\end{eqnarray*}
		The same estimate holds for $\| f_n - (f_n)_{\psi(t)} \|_1^2$. Hence $|I_n(0) - I_n(t)| \leq 2 \kappa_n d(t)$ for any $t \in G$.
		
    Since $I_n(0) = \lVert f_n^3 \rVert_1 \geq \lVert f_n \rVert_1^3 = \lVert f \rVert_1^3 \geq \delta^3$, it follows that
    \[
        I_n(t) \geq \delta^3 - 2 \kappa_n d(t) \text{ for all } t \in G.
    \]
    Note that $d(t) \leq \epsilon/(2\kappa_n)$ for $t$ in the set $D(\epsilon/(2 \kappa_n))$ and so $I_n(t) \geq \delta^3 - \epsilon$ in this set. 
    
    
    Let $\eta = \epsilon/(2 \kappa_n)$. In view of \cref{lem:kernel_a}, there exists a kernel $K$ supported on $D(\eta)$ such that $\lVert K \rVert_{\infty} \leq 1/\nu(\eta)$.
    We then have
    \[
        J'(f_{n}) := \int_G I_{n}(t) K(t) \, d \mu(t) \geq (\delta^3 - \epsilon) \int_{D(\eta)} K(t) \, d\mu(t) \geq \delta^3 - \epsilon.
    \]
    Letting $f_{st} = f_n$, $f_{er} = f_{n+1} - f_n$ and $f_{un} = f - f_{n+1}$, we obtain
    \begin{enumerate}
        \item $\lVert K \rVert_{\infty} \leq 1/\nu(\epsilon/(2\kappa_n)) \leq 1/\nu(\epsilon/(2 \kappa_M))$ (since $\nu$ is increasing on $(0, e^{1/2}/C_0) \supset (0, 1)$),

        \item $\lVert f_{st} \rVert_{\infty} = \lVert f_n \rVert_{\infty} = \lVert f * K_n \rVert_{\infty} \leq \lVert f \rVert_{\infty} \lVert K_n \rVert_1 \leq 1$, and $\lVert f_{er} \rVert_{\infty} \leq 2$, and $\lVert f_{un} \rVert_{\infty} \leq 2$,
        
        \item $\lVert f_{er} \rVert_2 = \lVert f_{n+1} - f_n \rVert_2 \leq \epsilon$,
        
        \item $\lVert \hat{f}_{un} \rVert_{\infty} \lVert K \rVert_{\infty} = \lVert \hat{f} - \hat{f}_{n+1} \rVert_{\infty} \lVert K \rVert_{\infty} < \eta_{n+1}/ \nu(\eta) \leq \epsilon$ because $\eta_{n+1} \leq \epsilon \nu(\epsilon/(2\kappa_n)) = \epsilon \nu(\eta).$
        
        \item $J'(f_{st}) = J'(f_n) \geq \delta^3 - \epsilon$.
    \end{enumerate}
    Our proof finishes. 
\end{proof}

\subsection{Proof of density results}
The goal of this section is to prove \cref{th:main-density} and \cref{th:roth-khintchine}. First we recall \cref{th:roth-khintchine} for the reader's convenience.

\begin{theorem*}[Khintchine-Roth theorem for compact abelian groups]  
Let $f: G \to [0, 1]$ be a measurable function with $\int_G f \, d \mu > \delta$. Let $\phi, \psi: G \to G$ be continuous homomorphisms such that $[G: \phi(G)], [G: \psi(G)]$ and $[G:(\phi-\psi)G]$ are finite. Then for every $\epsilon > 0$, there exists a constant $c_1$ that depends only on $\delta, \epsilon$ and the indexes above such that the set
\[
    B = \left\{t \in G: \int_G f(x) f(x+ \phi(t)) f(x + \psi(t)) \, d\mu(x) > \delta^3 - \epsilon \right\}
\]
has measure greater than $c_1$. As a consequence, there exists a constant $c_2$ that depends only on $\delta$ and indexes of $\phi(G), \psi(G), (\phi - \psi)(G)$ such that
\[
    J(f) := \iint_{G^2} f(x) f(x+ \phi(t)) f(x+ \psi(t)) d\mu(x) d\mu(t) > c_2.
\]
\end{theorem*}
\begin{proof}
    Fix $\epsilon > 0$ and let constant $C$, kernel $K$ and the decomposition $f = f_{st} + f_{er} + f_{un}$ be as found in \cref{prop:regularity_lemma}. Define
    \[
        J'(f) := \iint_{G^2} f(x) f(x+ \phi(t)) f(x+ \psi(t)) K(t) \ d\mu(x) d\mu(t)
    \]
    and 
    \[
        J'(f_{st}) := \iint_{G^2} f_{st}(x) f_{st}(x + \phi(t)) f_{st}(x + \psi(t)) K(t)\ d\mu(x) d\mu(t).
    \]
    
    Applying the decomposition $f = f_{st} + f_{er} + f_{un}$ and expanding $J'(f)$, we see that the difference $J'(f) - J'(f_{st})$ will have $26$ terms. The terms that contain $f_{er}$ can be bounded by $4 \epsilon$ since for $f_1, f_2, f_3 \in L^{\infty}(G)$,
    \begin{equation*}
        \iint_{G^2} f_1(x) f_2(x + \phi(t)) f_3(x + \psi(t)) K(t) \, d\mu(x) d\mu(t) \leq \max_{i} \lVert f_i \rVert_{\infty}^2 \max_{j} \lVert f_j \rVert_1 \lVert K \rVert_1.
    \end{equation*}
    
    On the other hand, in view of \cref{lem:bourgain_lemma_2} and a change of variables if necessary, the terms containing $f_{un}$ are bounded by $O(\lVert \hat{f}_{un} \rVert_{\infty} \lVert \widehat{K} \rVert_{1})$ which is $O(\epsilon)$ thanks to the properties of the decomposition. (The implicit constant may depend on the index $[G:(\phi - \psi)(G)]$ now due to this change of variables.)
    Therefore,
    \begin{equation}
    \label{eq:J_f'_1}
        J'(f) > J'(f_{st}) - O(\epsilon) > \delta^3 - c_0 \epsilon
    \end{equation}
    where the constant $c_0$ depends only on the indexes of $\phi(G), \psi(G)$ and $(\phi - \psi)(G)$ in $G$. 
    
    Define
    \begin{equation*}
        I_f(t) = \int_G f(x) f(x+ \phi(t)) f(x+ \psi(t)) \, d\mu(x)
    \end{equation*}
    and 
    \begin{equation*}
        B = \{t \in G: I_f(t) > \delta^3 - 2 c_0 \epsilon \}. 
    \end{equation*}
    We then have
    \begin{align}
    \label{eq:J_f'_2}
    \begin{aligned}
        J'(f) = \int_G I_f(t) K(t) \, d\mu(t) = \int_B I_f(t) K(t) \, d\mu(t) + \int_{G \setminus B} I_f(t) K(t) \, d\mu(t) \leq \\ \int_B K(t) \, d\mu(t) + (\delta^3 - 2 c_0 \epsilon) \int_{G \setminus B} K(t) \, d\mu(t) \leq 
        \lVert K \rVert_{\infty} \mu(B) + (\delta^3 - 2 c_0 \epsilon).
    \end{aligned}
    \end{align}
    Combining \eqref{eq:J_f'_1} and \eqref{eq:J_f'_2}, we deduce that 
    \begin{equation*}
        \mu(B) > c_0 \epsilon/\lVert K \rVert_{\infty}.
    \end{equation*}
    Letting $c_1 = c_0 \epsilon/\lVert K \rVert_{\infty}$, we obtain the first part of the theorem. 
    
    Now we have
    \begin{equation*}
        J(f) = \int_G I_f(t) \, d\mu(t) > (\delta^3 - 2 c_0 \epsilon) c_1.
    \end{equation*}
    Letting $c_2 = c_1(\delta^3 - 2c_0 \epsilon)$, we obtain the second part of the theorem. 
\end{proof}

In order to prove \cref{th:main-density}, we need the following proposition. For our future applications, we will state and prove a slightly more general version than what is necessary.


\begin{proposition}
\label{prop:density-bohr-stronger}
Suppose $\phi, \psi: G \rightarrow G$ are continuous homomorphisms such that $\phi(G), \psi(G), (\phi-\psi)(G)$ have finite index in $G$. Let $f : G \rightarrow [0,1]$ such that $\int_G f \, d\mu = \delta >0$. For $w \in G$, define
\[
    R(w) = \iint_{G^2} f(x + w) f(x + \phi(y)) f(x+\psi(y))  \, d\mu(x) d\mu(y)
\]
Then there are $c, k, \eta >0$ depending only on $\delta$ and the indexes above such that the set
 $\{ w \in G: R(w) > c\}$ contains a Bohr-$(k,\eta)$ set. 
\end{proposition}

\begin{proof}
By \cref{lem:bourgain_lemma_2}, we have
\[
    |R(w) - R(0)| \ll \lVert \hat{f} - \widehat{f_w} \rVert_{\infty}
\]
where implicit constant depends only on the indexes of $\phi(G), \psi(G)$ in $G$. 
By \cref{th:roth-khintchine}, we know that $R(0) > c$ for some constant $c>0$ depending on the indexes $[G: \phi(G)]$, $[G: \psi(G)]$, $[G: (\phi - \psi)(G)]$ and $\delta$. It follows that there exists a constant $c'$ such that $R(w) > c/2$ if 
\begin{equation}
    \lVert \hat{f} - \widehat{f_w} \rVert_{\infty} < c'.
\end{equation}
\cref{lem:translate} implies that the set of such $w$ contains a Bohr-$(k,\eta)$ set, where $k$ and $\eta$ depend only on $c'$.
\end{proof}

We can now formulate and prove our main theorem for sets of positive measure.

\begin{theorem} \label{th:main-density-true}
Let $\phi_1, \phi_2, \phi_3, \psi_1, \psi_2: G \rightarrow G$ be continuous homomorphisms satisfying the following:
\begin{enumerate}
\item $\phi_1 + \phi_2 + \phi_3=0$,
\item $\phi_1 \circ \psi_1 = \phi_2 \circ \psi_2$,
\item $\phi_3(G), \psi_1(G), \psi_2(G), (\psi_1 + \psi_2)(G)$ have finite index in $G$.
\end{enumerate}
The for any measurable set $A \subset G$, $\mu(A) = \delta >0$, the set $\phi_1(A) + \phi_2(A) + \phi_3(A)$ contains a Bohr-$(k, \eta)$ set, where $k$ and $\eta$ depend only on $\delta$ and the indexes above. 
\end{theorem}
\begin{proof}
Applying \cref{prop:density-bohr-stronger} for $f=1_A$ and $\psi_1, -\psi_2$ in place of $\phi$ and $\psi$, we see that there exists a Bohr-$(k,\eta)$ set $B$ such that for all $w \in B$, there are $x, y \in G$ such that
\[
x+w , x+\psi_1(y) \textup{ and } x-\psi_{2}(y) \in A.
\]
Note that
\[
\phi_3(x + w) + \phi_1(x+\psi_1(y)) + \phi_2(x-\psi_2(y)) = \phi_3(w)
\]
and so that $\phi_1(A) + \phi_2(A) + \phi_3(A) \supseteq \phi_3(B)$.
Our theorem then follows from \cref{lem:bohr-homo2}.
\end{proof}

\cref{th:main-density} is now a special case of \cref{th:main-density-true} with $\psi_1 = \phi_2$ and $\psi_2 = \phi_1$. Note that the condition of \cref{th:main-density-true} is met because $\phi_1 + \phi_2 = - \phi_3$ and so
\[
    [G: (\psi_1 + \psi_2)(G)] = [G:(\phi_1 + \phi_2)(G)] = [G:(-\phi_3)(G)] = [G: \phi_3(G)]
\]
which is finite.

\begin{remark} \label{rk:auto}
Here we explain why the condition that $\phi_1, \phi_2, \phi_3$ commute in \cref{th:main-density} is not necessary if one of the $\phi_i$'s is an automorphism (see \cref{remark:density_commutative}). Without loss of generality, assume $\phi_1$ is an automorphism. Let $\psi_1 = \phi_1^{-1} \circ \phi_2$ and $\psi_2 = \textup{Id}$. Then the first two conditions of \cref{th:main-density-true} are satisfied. As for the third condition, we have $\psi_1(G) = \phi_1^{-1} \circ \phi_2(G)$ and $(\psi_1 + \psi_2)(G) = \phi_1^{-1} \circ (\phi_2 + \phi_1) (G)$. Both of these have finite index in $G$ by \cref{lem:composition}.
\end{remark}

\section{Bohr sets in sumsets in number fields and function fields}

\label{sec:z_and_field}

In this section we prove Theorems \ref{th:main-kr}, \ref{th:nf} and \ref{th:ff} using a strategy similar to Bergelson and Ruzsa's proof of \cref{th:br}. To prove \cref{th:br}, one could embed $A \cap [N]$ naturally in $\Z_N$, and invoke the counting result (for example, \cref{th:roth-khintchine}) in $\Z_N$. However, one has to deal with the ``wraparound effect'': A solution to $s_1 x + s_2 y + s_3z =0$ in $\Z_N$ does not necessarily correspond to a solution in $\Z$. To overcome this issue, Bergelson and Ruzsa embedded $A \cap [N]$ in $\Z_{N'}$ for some $N' \gg N$. Then $A \cap [N]$ remains dense in $\Z_{N'}$ and a solution in $\Z_{N'}$ found in $A \cap [N]$ is now a solution in $\Z$.

For partitions, the corresponding counting result would be \cref{th:counting-partition}. However, if this theorem were applied directly, we would have a partition of the whole group $\Z_{N'}$ which again causes the wrap-around effect. To avoid this problem, we need to modify \cref{th:counting-partition} so that it allows for partitions of a subset $[-N, N] \subset \Z_{N'}$ instead of the whole group. 


\begin{proposition}\label{prop:counting-brauer}
For any $k, \ell, r>0$, there is a constant $c(k, \ell, r) > 0$ such that the following holds: For sufficiently large $N$, if $[-N, N] = \bigcup_{i=1}^r A_i$, then for some $1 \leq i \leq r$, we have
\[
\sum_{|x|, |y| \leq N} 1_{A_i}(\ell y) 1_{A_i}(x) 1_{A_i}(x+y) \cdots 1_{A_i}(x + ky) \geq c(k, \ell, r) N^2.
\]
(Here $1_{A_i}(n) = 0$ if $n \not \in [-N, N]$.)
\end{proposition}
\begin{remark}
\cref{prop:counting-brauer} also follows from Frankl-Graham-R\"odl \cite[Theorem 1]{fgr}, but our proof shows that it is directly in line with \cref{th:counting-partition}. Furthermore, our proof easily generalizes to other rings such as $\Z[i]$ and $\Fq[t]$.
\end{remark}
\begin{proof}
We apply \cref{th:mpc-hom} with $\psi(y) = \ell y$ and $\phi_j(y) = jy$ for $1 \leq j \leq k$. Then there exist $m$ and $n$ depending only on $r$ and $k$ such that for any $r$-coloring of $S_m(x_1, \ldots, x_n)$, there are $x$ and $y$ of nonempty and disjoint support such that the configuration
\[
\{ \ell y, x, x+y, \ldots, x+ky\}
\] 
is monochromatic.

Note that elements of $S_m(x_1, \ldots, x_n)$ are all linear forms in $x_1, \ldots, x_n$ with bounded integer coefficients. Therefore, there is a constant $c > 0$ such that for all $x_1, \ldots, x_n \in [-cN, cN]$, $S_m(x_1, \ldots, x_n) \subset [-N, N]$. Because of this inclusion, for each $(x_1, \ldots, x_n) \in [-cN, cN]^n$, the coloring $[-N, N] = \bigcup_{i=1}^r A_i$ naturally induces an $r$-coloring of $S_m(x_1, \ldots, x_n)$. 
By \cref{th:mpc-hom}, $S_m(x_1, \ldots, x_n)$ contains a monochromatic configuration of the form $\{ \ell y, x, x+y, \ldots, x+ky\}$. There are $\gg N^n$ monochromatic configurations in $[-N, N]$ arising in this way. However, a configuration may come from many different sets $S_m(x_1, \ldots, x_n)$. 
We will show that the number of tuples $(x_1, \ldots, x_n)$ giving rise to the same configuration $\{ \ell y, x, x+y, \ldots, x+ky\}$ is $\ll N^{n-2}$.

Indeed, let $I, J$ be disjoint nonempty subsets of $[n]$ such that $x$ and $y$ are linear combinations with bounded coefficients of $(x_i)_{i \in I}$ and $(x_j)_{j \in J}$, respectively. For fixed $I$ and $J$, the number of choices for $(x_i)_{i \in I}$ is $\ll N^{|I|-1}$, since any choice of $(|I|-1)$ of the $x_i$'s gives at most one value for the remaining $x_i$. For the same reason, the number of choices for $(x_j)_{j \in J}$ is $\ll N^{|J|-1}$. Since there are finitely many pairs $(I, J)$, we see that the number of $(x_1, \ldots, x_n)$ that give rise to $\{ \ell y, x, x+y, \ldots, x+ky\}$ is $\ll N^{n-2}$. Hence the number of monochromatic configurations in $[N]$ is $\gg N^2$,  and we are done.
\end{proof}
Our next statement is essentially a diagonalization argument.

\begin{proposition} \label{prop:counting-diagonal}
Let $\mathcal{P}$ denote an arbitrary partition $\Z = \bigcup_{i=1}^r A_i$. Then there exists some $1 \leq i \leq r$ with the following property: For every $\ell \geq 0$, there is a constant $c(\ell, \mathcal{P}) > 0$ such that 
\[
\sum_{|x|, |y| \leq N} 1_{A_i}(y) 1_{A_i}(x) 1_{A_i}(x+ \ell y) \geq c(\ell, \mathcal{P}) N^2
\]
for infinitely many $N \in \N$.
\end{proposition}

\begin{proof}
Invoking \cref{prop:counting-brauer}, for each $k \in \N$, there is $i=f(k)$ such that for infinitely many $N$, we have
\[
\sum_{|x|, |y| \leq N} 1_{A_i}(y) 1_{A_i}(x) 1_{A_i}(x+y) \cdots 1_{A_i}(x + ky) \geq c(k,1,r) N^2.
\]
Hence there exist an $i \in \{1, \ldots, r\}$ and an infinite set $K$ such that $f(k) = i$ for all $k \in K$. 

Let $\ell$ be arbitrary and pick $k \in K, k\geq \ell$. We have, for infinitely many $N$,
\begin{multline*}
\sum_{|x|, |y| \leq N} 1_{A_{i}}(y) 1_{A_{i}}(x) 1_{A_{i}}(x+ \ell y) \\
\geq \sum_{|x|, |y| \leq N} 1_{A_{i}}(y) 1_{A_{i}}(x) 1_{A_{i}}(x+y) \cdots 1_{A_{i}}(x + ky) \geq c(k,1,r) N^2,
\end{multline*}
thus proving the desired claim.
\end{proof}

\begin{remark} In the proof above, we do not have any control on $c(k,1, r)$ since we do not have control on $k$. As a result, the constant $c(\ell, \mathcal{P})$ above depends on the partition. It is interesting to see if this dependence is indeed necessary.
\end{remark}

We can now prove \cref{th:main-kr}.

\begin{proof}[Proof of \cref{th:main-kr}(a)]
Let $\Z = \bigcup_{i=1}^r A_i$ be an arbitrary partition and $s_1, s_2 \in \Z \setminus \{ 0 \}$. 
Without loss of generality, we assume $s_1, s_2 >0$. 
For a set $A \subset \Z$ and $N >0$, we write $A^{(N)}$ to denote $A \cap [-N,N]$.

By Proposition \ref{prop:counting-brauer}, there exist $i \in [r]$ and an infinite set $\calN$ such that
\begin{equation}\label{eq:schur1}
\sum_{|x|, |y| \leq N} 1_{A_i^{(N)}}(s_1 y) 1_{A_i^{(N)}}(x) 1_{A_i^{(N)}}(x+ s_2y)  \geq c N^2
\end{equation}
for any $N \in \calN$, where $c>0$ is a constant independent of $N$.

Let $N'$ be the smallest odd integer greater than $(2s_1 + s_2 + 1)N$. We identify $\Z_{N'}$ with $[ - \frac{N'-1}{2}, \frac{N'-1}{2}]$. Then \eqref{eq:schur1} implies that

\begin{equation}\label{eq:schur2}
\sum_{x, y \in \Z_{N'}} 1_{A_i^{(N)}}(s_1 y) 1_{A_i^{(N)}}(x) 1_{A_i^{(N)}}(x+ s_2y)  \geq c' N'^2
\end{equation}
for some constant $c'>0$ independent of $N$.
Define 
\[
R(w) = \sum_{x, y \in \Z_{N'}} 1_{A^{(N)}_i}(s_1 y) 1_{A^{(N)}_i}(x+w) 1_{A^{(N)}_i}(x+ s_2 y).
\]
Then by the same argument as the proof of \cref{prop:partition-bohr-stronger}, the set $\{w \in \Z_{N'}: R(w) > 0 \}$ contains a Bohr-$(k, \eta)$ set in $\Z_{N'}$, where $k$ and $\eta$ are independent of $N$. Note that $R(w)>0$, implies there are $a, a', a'' \in A^{(N)}_i$ and $x, y \in \Z_{N'}$ such that
\[
s_1 y \equiv a, \quad x+w \equiv a', \quad \textup{and} \quad x+s_2 y \equiv a'' \pmod{N'}.
\]
Therefore,
\[
s_1 w = s_1 (x+w) - s_1 (x + s_2 y) + s_2 (s_1 y) \equiv s_1 a' - s_1 a'' + s_2 a \pmod{N'}.
\]
If $|w| \leq N$ then this congruence is an equality in $\Z$ thanks to the way we choose $N'$ and the fact that $|a|, |a'|, |a''| \leq N$. We have thus proved that, for each $N \in \calN$, there exist $x_1, \ldots, x_k \in [-\frac{N'-1}{2}, \frac{N'-1}{2}]$ such that
\[
(s_1 A_i - s_1 A_i + s_2 A_i)/s_1 \supset [-N, N] \cap \left\{ w \in \Z: 
\left| e \left(\frac{ x_j w}{N'} \right) -1 \right| < \eta, \quad \forall j=1, \ldots, k\right\}.
\]
Here we are using the notation $A/c$ defined in \eqref{eq:ring2}.

Taking $N \to \infty$, $N \in \mathcal{N}$, and passing to a subsequence if necessary,
the sequence $(\frac{x_1}{N'}, \ldots, \frac{x_k}{N'})$ converges to a point $(\alpha_1, \ldots, \alpha_k)$ in $(\R/\Z)^k$. Hence,
\begin{equation*} \label{eq:bohrset1}
(s_1 A_i - s_1 A_i + s_2 A_i)/s_1 \supset \left\{ w \in \Z: 
\left| e\left(  \alpha_j w \right) -1 \right| < \frac{\eta}{2}, \quad \forall j=1, \ldots, k\right\}.
\end{equation*}
This implies that
\begin{equation*} \label{eq:bohrset2}
 s_1 A_i - s_1 A_i + s_2 A_i
\supset  \left\{ n \in \Z: 
\left| e\left(  \frac{\alpha_j n}{s_1}  \right) -1 \right| < \frac{\eta}{2}, \quad \forall j=1, \ldots, k \right\} \cap s_1 \Z.
\end{equation*}
Since $s_1 \Z$ is a Bohr set and the intersection of two Bohr sets is a Bohr set, our proof finishes.
\end{proof}

\begin{proof}[Proof of \cref{th:main-kr}(b)] We proceed similarly to part (a), using \cref{prop:counting-diagonal} instead of \cref{prop:counting-brauer}. Let $\mathcal{P}$ be an arbitrary partition $\Z = \bigcup_{i=1}^r A_i$. Let $i$ be given by \cref{prop:counting-diagonal}. Let $s \in \Z \setminus \{0\}$ be arbitrary. Since $A - A + sA = - (A - A - sA)$, $A - A + sA$ contains a Bohr-$(k, \eta)$ set if and only if $A - A - sA$ contains a Bohr-$(k, \eta)$ set. Thus, without loss of generality, we can assume $s > 0$. There is an infinite set $\calN_s \subset \N$ such that for any $N \in \calN_s$, we have
\begin{equation} 
\sum_{|x|, |y| \leq N} 1_{A_i^{(N)}}(y) 1_{A_i^{(N)}}(x) 1_{A_i^{(N)}}(x+ s y) \geq c(s, \mathcal{P}) N^2,
\end{equation}
for some constant $c(s, \mathcal{P})>0$ independent of $N$. Note that
\[
w = (w+x) - (x+sy) + sy.
\]
The rest is identical to part (a).
\end{proof}






\subsection{Sumsets in \texorpdfstring{$\Z[i]$}{Z[i]}} 

Even though the corresponding tori in the cases of $\Z[i]$ and $\Fq[t]$ are slightly different from $\Z$, the general approaches are very similar. Therefore, we will be brief and highlight only the differences.

The following proposition is needed for the proof of \cref{th:nf}(b,c). We omit its proof since it is identical to the ones of Propositions \ref{prop:counting-brauer} and \ref{prop:counting-diagonal}.


\begin{proposition}\label{prop:counting-brauer-nf}\

\begin{enumerate}[label=(\alph*)]
    \item Let $b, a_1, \ldots, a_k \in \Z[i]$ and $r>0$. There is a constant $c = c(b, a_1, \ldots, a_k, r)$ $>0$ such that the following holds: For $N$ sufficiently large, if $[-N, N]^2 = \bigcup_{j=1}^r A_j$, then for some $1 \leq j \leq r$, we have
\[
\sum_{x, y \in [-N, N]^2} 1_{A_j}(by) 1_{A_j}(x) 1_{A_j}(x+a_1 y) \cdots 1_{A_j}(x + a_k y) \geq c N^4.
\]
(Here $1_{A_j} = 0$ outside of $[-N, N]^2$.)

\item Let $\mathcal{P}$ denote an arbitrary partition $\Z[i] = \bigcup_{j=1}^r A_j$. Then there exists some $1 \leq j \leq r$ with the following property: For each $\ell \in \Z[i]$, there is a constant $c(\ell, \mathcal{P}) > 0$ such that 
\[
\sum_{x, y \in [-N, N]^2} 1_{A_j}(y) 1_{A_j}(x) 1_{A_j}(x+ \ell y) \geq c(\ell, \mathcal{P}) N^4.
\]
for infinitely many $N \in \N$.
\end{enumerate}
\end{proposition}
\begin{proof}[Proof of Theorem \ref{th:nf} (a)] Suppose $A \subset \Z[i]$ has $\overline{d}(A) = \delta >0$. Then for infinitely many $N$, we have $|A^{(N)}|  \geq \delta N^2$, where $A^{(N)} = A \cap [-N, N]^2$.

Let $N' = 2(|s_1| + |s_2| + |s_3|) N + 1$. We identify $[-\frac{N'-1}{2}, \frac{N'-1}{2}]^2$ with $\Z_{N'} \times \Z_{N'}$. By \cref{th:main-density}, the set $s_1 A^{(N)} + s_2 A^{(N)} + s_3 A^{(N)}$ contains a Bohr set in $\Z_{N'} \times \Z_{N'}$, which is of the form
\[
\left\{ (w,v) \in \Z_{N'} \times \Z_{N'} : \left| e\left( \frac{wx_j + vy_j}{N'} \right) -1 \right| < \eta,  \quad \forall j=1, \ldots, k\right\}
\]
for some $x_1, \ldots, x_k, y_1, \ldots, y_k \in [-\frac{N'-1}{2}, \frac{N'-1}{2}]$, where $k$ and $\eta$ depend only on $\delta$ and $s_1, s_2, s_3$.

If $(w, v)$ is in the Bohr set above and $|w|, |v| \leq N$, then there exist $a, a', a'' \in \AN$ such that
\[
(w,v) = s_1 a + s_2 a' + s_3 a'',  
\]
where the equality is in $\Z[i]$ and not just in $\Z_{N'} \times \Z_{N'}$. Hence,
\[
s_1 A + s_2 A + s_3 A \supset [-N, N]^2 \cap \left\{ (w,v) \in \Z[i] : \left| e\left( \frac{wx_j + vy_j}{N'} \right) -1 \right| < \eta,  \quad \forall j=1, \ldots, k\right\}.
\]
Letting $N$ go to infinity along some subsequence, we have that
\[
s_1 A + s_2 A + s_3 A \supset \left\{ (w,v) \in \Z[i] : \left| e\left( w \alpha_j + v\beta_j \right) -1 \right| < \frac{\eta}{2},  \quad \forall j=1, \ldots, k\right\},
\]
where $(\alpha_1, \ldots, \alpha_k, \beta_1, \ldots, \beta_k)$ is a limit point of $( \frac{x_1}{N'}, \ldots, \frac{x_k}{N'}, \frac{y_1}{N'}, \ldots, \frac{y_k}{N'} )$, and we are done.
\end{proof}

\begin{proof}[Proof of Theorem \ref{th:nf}(b)]
Using Proposition \ref{prop:counting-brauer-nf}(a) and arguing similarly to the proof of Theorem \ref{th:main-kr}(a), we see that for some $1 \leq j \leq r$, for infinitely many $N$, we have
\[
(s_1 A_j - s_1 A_j + s_2 A_j)/s_1 \supset [-N, N]^2 \cap \left\{ (w,v) \in \Z[i] : \left| e\left( \frac{wx_j + vy_j}{N'} \right) -1 \right| < \eta,  \quad \forall j=1, \ldots, k\right\}.
\]
Letting $N$ go to infinity, we have
\[
(s_1 A_j - s_1 A_j + s_2 A_j)/s_1 \supset \left\{ (w,v) \in \Z[i] : \left| e\left( w \alpha_j + v\beta_j \right) -1 \right| < \frac{\eta}{2},  \quad \forall j=1, \ldots, k\right\},
\]
where $(\alpha_1, \ldots, \alpha_k, \beta_1, \ldots, \beta_k)$ is a limit point of $( \frac{x_1}{N'}, \ldots, \frac{x_k}{N'}, \frac{y_1}{N'}, \ldots, \frac{y_k}{N'} )$. Note that
\[
w \alpha_j + v \beta_j = \Re( (w + iv) (\alpha_j -i\beta_j))
\]
and hence,
\begin{equation}\label{eq:real_part}
s_1 A - s_1 A + s_2 A \supset \left\{ z \in \Z[i] : 
\left| e\left( \Re \left( z \frac{\alpha_j - i \beta_j}{s_1} \right) \right) -1 \right| < \frac{\eta}{2},  \quad \forall j=1, \ldots, k
\right\} \cap s_1 \Z[i]. 
\end{equation}
Since $\Re: \Z[i] \to \R$ is a homomorphism, for any $a \in \C$, $z \mapsto e(\Re(az))$ is a continuous homomorphism from $\Z[i] \to S^1$. Therefore, the right hand side of \eqref{eq:real_part} is a Bohr set by \cref{lem:subgroup-bohr}.
\end{proof}

The proof of Theorem \ref{th:nf}(c) is similar to part (b), using Proposition \ref{prop:counting-brauer-nf}(b) instead of Proposition \ref{prop:counting-brauer-nf}(a).

\subsection{Sumsets in \texorpdfstring{$\Fq[t]$}{Fq[t]}} Let $p$ be a prime and $q$ be a power of $p$. First, let us introduce some standard facts about $\Fq[t]$. 
Let $\K=\Fq(t)$ be the field of fractions of $\Fq[t]$. For $f/g \in \K$ we define $|f/g|=q^{\deg(f)-\deg(g)}$ and $|0| =0$. The completion of $\K$ with respect to $|\cdot|$ is 
$\Kinf=\Fqt = \left\{ \sum_{i=-\infty}^n a_i t^i: a_i \in \Fq, n \in \Z \right\}$. Let $\T_q = \left\{ \sum_{i=-\infty}^{-1} a_i t^i: a_i \in \Fq \right\}$. Then $\Fq[t], \K, \Kinf, \T_q$ are the analogs of $\Z, \Q, \R$ and $\R/\Z$, respectively.

For $x \in \Fq$, we write $e_q(x) = e\left( \frac{\textup{Tr} (x)}{p} \right)$, where $\textup{Tr}: \Fq \rightarrow \Fp$ is the trace map.\footnote{That is, $\textup{Tr}(x)$ is the trace of the $\F_p$-linear map $y \mapsto xy$ from $\Fq$ to $\Fq$, when $\Fq$ is viewed as a $\F_p$-vector space. In particular, $\textup{Tr}(x) \in \F_p$.} It can be checked that $x \mapsto e_q(ax)$ (where $a \in \Fq$) are all the additive characters of $\Fq$.

If $\alpha = \sum_{i=-\infty}^n a_i t^i \in \Kinf$, we write $(\alpha)_{-1} = a_{-1}$ and $E(\alpha) = e_q(a_{-1})$. It can be checked that $f \mapsto E( f \alpha)$, where $\alpha \in \T_q$, are all the continuous characters of $\Fq[t]$. This also shows that $\T_q$ is the dual of $\Fq[t]$.

Any Bohr set $B$ in $\Fq[t]$ is of the form
\[
B = \left\{ f \in \Fq[t]: \left| E(f \alpha_i) -1 \right| < \eta \textup{ for } i=1, \ldots, k \right\},
\]
where $\alpha_1, \ldots, \alpha_k \in \T_q$. If $\eta < |e(1/p)-1|$ then 
\[
B = \left\{ f \in \Fq[t]: \textup{Tr}((f \alpha_i)_{-1}) =0 \textup{ for } i=1, \ldots, k \right\}.
\]
This is an $\Fp$-subspace and not necessarily an $\Fq$-subspace. However, it contains the $\Fq$-subspace
\[
\left\{ f \in \Fq[t]: (f \alpha_i)_{-1} =0 \textup{ for } i=1, \ldots, k \right\}.
\]
We write $G_N = \{ f\in \Fq[t]: \deg(f) < N\}$. For a set $A \subset \Fq[t]$, we write $\AN$ for $A \cap G_N$. 

Using the same arguments as in Propositions \ref{prop:counting-brauer} and \ref{prop:counting-diagonal}, we can prove the following:

\begin{proposition}\label{prop:counting-brauer-ff}\

\begin{enumerate}[label=(\alph*)]
    \item Let $b, a_1, \ldots, a_k \in \Fq[t]$ and $r>0$. There is a number $c=c(q, b, a_1, \ldots, a_k, r) >0$ such that the following holds. For $N$ sufficiently large, if $G_N = \bigcup_{j=1}^r A_i$, then for some $1 \leq i \leq r$, we have
\[
\sum_{x, y \in G_N} 1_{A_i}(by) 1_{A_i}(x) 1_{A_i}(x+a_1 y) \cdots 1_{A_i}(x + a_k y) \geq c q^{2N}.
\]
(Here we define $1_{A_i} = 0$ outside of $G_N$.)

\item Let $\mathcal{P}$ denote an arbitrary partition $\Fq[t] = \bigcup_{i=1}^r A_i$. Then there exists some $1 \leq i \leq r$ with the following property: For each $\ell \in \Fq[t]$, there is a constant $c(\ell, \mathcal{P}) > 0$ such that 
\[
\sum_{x, y \in G_N} 1_{A_i}(y) 1_{A_i}(x) 1_{A_i}(x+ \ell y) \geq c(\ell, \mathcal{P}) q^{2N}
\]
for infinitely many $N \in \N$.
\end{enumerate}
\end{proposition}

\begin{proof}[Proof of Theorem \ref{th:ff}]
We will sketch the proof of \cref{th:ff}(b). Parts (a) and (c) can be proved along the same lines.

Let $\Fq[t] = \bigcup_{i=1}^r A_i$ be an arbitrary partition and $s_1, s_2 \in \Fq[t] \setminus \{ 0 \}$. 
By Proposition \ref{prop:counting-brauer-ff}(a), we know that there exist $1 \leq i \leq r$ and an infinite set $\calN$ such that
\begin{equation}\label{eq:schur3}
\sum_{x, y \in G_N} 1_{A_i^{(N)}}(s_1 y) 1_{A_i^{(N)}}(x) 1_{A_i^{(N)}}(x+ s_2y)  \gg q^{2N}
\end{equation}
for each $N \in \calN$.

Let $N'= \max(\deg s_1, \deg s_2) + N$. We fix a polynomial $P_{N'} \in \Fq[t]$ of degree $N'$.  We identify $G_{N'}$ with $\Fq[t] / (P_{N'})$, the latter being a ring and playing the role of $\Z_{N'}$ in the proof of Theorem \ref{th:main-kr}. Then \eqref{eq:schur3} implies that
\begin{equation}\label{eq:schur4}
\sum_{x, y \in \Fq[t] / (P_{N'})} 1_{A_i^{(N)}}(s_1 y) 1_{A_i^{(N)}}(x) 1_{A_i^{(N)}}(x+ s_2y)  \gg q^{2N'}.
\end{equation}

Arguing similarly to the proof of Theorem \ref{th:main-kr}(a), we find that
\[
(s_1 A_i - s_1 A_i + s_2 A_i)/s_1 \supset G_N \cap \left\{ w \in \Fq[t]: 
(w \frac{x_i}{P_{N'}})_{-1} = 0, \quad \forall j=1 \ldots, k\right\},
\]
for some $x_1, \ldots, x_k \in G_{N'}$.

Letting $N \rightarrow \infty$ and using compactness of $\T_q$, we have
\begin{equation*} 
(s_1 A_i - s_1 A_i + s_2 A_i)/s_1 \supset \left\{ w \in \Fq[t]: 
(w \alpha_i)_{-1} = 0, \quad \forall j=1, \ldots, k\right\}
\end{equation*}
for some $\alpha_1, \ldots, \alpha_k \in \T_q$. Therefore,
\[
s_1 A_i - s_1 A_i + s_2 A_i \supset \left\{ f \in \Fq[t]: 
(f \frac{\alpha_i}{s_1})_{-1} = 0, \quad \forall j=1, \ldots, k\right\} \cap s_1 \Fq[t],
\]
which is clearly an $\Fq$-subspace of bounded codimension.
\end{proof}

\section{Open questions}
\label{sec:open_question}

\cref{th:main-kr}(b) says that in any partition $\Z = \bigcup_{i=1}^r A_i$, there exists an $i \in \{1, \ldots, r\}$ such that $A_i - A_i + sA_i$ contains a Bohr set for every $s \in \Z \setminus \{0\}$. Inspired by Katznelson and Ruzsa's question, \cref{th:main-kr}(b) naturally gives rise to the following question.

\begin{question}
\label{ques:B+sA}
Suppose $A \subseteq \Z$ does not contain a Bohr set and $B \subseteq \Z$ such that $B + sA$ contains a Bohr set for every $s \in \Z \setminus \{0\}$. Must it be true that $B$ contains a Bohr set?
\end{question}
A positive answer to \cref{ques:B+sA} would lead to a resolution of Katznelson-Ruzsa's question. However, it is likely that the answer to \cref{ques:B+sA} is negative.

As mentioned in the introduction, we do not know whether the commuting conditions in \cref{th:main-partition} and \cref{th:main-density} can be removed entirely or  not. It is interesting to answer the following. 


\begin{question} 
Can the commuting conditions in \cref{th:main-partition} and \cref{th:main-density} be removed?
\end{question}

\section*{Acknowledgments} 
We thank Vitaly Bergelson and John Griesmer for many helpful conversations on sets of recurrence, Bohr sets and related topics. We also thank the anonymous referee for a very thorough reading of the paper and numerous insightful comments and suggestions which helped to improve its presentation.
The second author was partially supported by National Science Foundation Grants DMS-1702296, DMS-2246921 and a travel gift from the Simons Foundation.

\bibliographystyle{amsplain}


\begin{dajauthors}
\begin{authorinfo}[anhle]
    Anh N. Le\\
    University of Denver\\
    Denver, CO 80210, USA\\
  anh.n.le@du.edu \\
\end{authorinfo}

\begin{authorinfo}[thoangle]
  Th\'ai Ho\`ang L\^e\\
  University of Mississippi\\
    University, MS 38677, USA\\
  leth@olemiss.edu \\
\end{authorinfo}
\end{dajauthors}

\end{document}